\documentclass[11 pt, leqno, hidelinks]{amsart}%
\usepackage{amsfonts}%
\usepackage{amsmath}%

\usepackage{amscd}
\usepackage{amsthm}
\usepackage{epsfig}
\usepackage{amstext}
\usepackage[all]{xy}
\usepackage{graphicx}%
\usepackage{amssymb}%
\usepackage{graphicx}
\usepackage{a4wide}
\usepackage{comment}
\usepackage{url,amsmath,amsthm,amssymb,mathtools}
\usepackage{hyperref}
\usepackage[margin=1.25in]{geometry}
\usepackage{xcolor}
\usepackage{comment}
\usepackage{mathbbol}

\usepackage{color}
\usepackage[noadjust]{cite}

\usepackage{makecell}

\newtheorem{theorem}{Theorem}[section]
\theoremstyle{plain}

\newtheorem{corollary}[theorem]{Corollary}

\newtheorem{lemma}[theorem]{Lemma}

\newtheorem{proposition}[theorem]{Proposition}
\newtheorem{quest}[theorem]{Question}
\newtheorem{remark}[theorem]{Remark}

\newtheorem{thmx}{Theorem}

\newtheorem{corx}[thmx]{Corollary}

\numberwithin{equation}{section}

\newcommand{\C}{\mathbb{C}}

\newcommand{\G}{\mathcal{G}}

\newcommand{\N}{\mathbb{N}}

\newcommand{\Z}{\mathbb{Z}}
\newcommand{\GG}{\mathcal{G}}
\newcommand{\tr}{\textup{tr}}
\newcommand{\Rep}{\textup{Rep}}
\newcommand{\Irr}{\textup{Irr}}
\newcommand{\Prim}{\textup{Prim}}
\newcommand{\Ker}{\textup{Ker}}

\newcommand{\supp}{\textup{supp}}

\newcommand{\id}[0]{\operatorname{id}}
\renewcommand{\ker}[0]{\operatorname{ker}}
\newcommand{\rank}[0]{\operatorname{rank}}

\newcommand{\Hil}{\mathsf{H}}

\newcommand{\Kil}{\mathsf{K}}

\begin{document}
\author{Tatiana Shulman}
\address{Department of Mathematical Sciences, Chalmers University of Technology, Chalmers tvärgata 3, 412 96 Gothenburg, Sweden}
\email{tatshu@chalmers.se}

\author{Adam Skalski}
\address{Institute of Mathematics of the Polish Academy of Sciences,
		ul.~\'Sniadeckich 8, 00--656 Warszawa, Poland}
\email{a.skalski@impan.pl}

\title[RFD for groupoids and crossed products]{RFD property for groupoid C*-algebras of amenable groupoids and for crossed products  by amenable actions}

\begin{abstract} By Bekka's theorem the group C*-algebra of an amenable group $G$ is residually finite dimensional (RFD) if and only if $G$ is  maximally almost periodic (MAP). We generalize this result in two directions of dynamical flavour. Firstly, we completely characterize the RFD property for crossed products by amenable actions of discrete groups on C*-algebras in terms of the action. The characterisation can be formulated in various terms,  such as primitive ideals, (pure) states and approximations of representations, and the latter can be viewed as a dynamical version of Exel-Loring characterization of RFD C*-algebras.
%The result leads among other consequences to a characterization of when a semidirect product by an amenable group has RFD full C*-algebra.
As byproduct of our methods we characterize the property FD of Lubotzky and Shalom for semidirect products by amenable groups and obtain characterizations of the properties MAP  and RF for general semidirect products of groups. These descriptions allow us to obtain the properties  MAP, RF, RFD and FD for various new examples and generalize some results of Lubotzky and Shalom.

Secondly,  as another generalization of Bekka's theorem, we
provide a sufficient condition and a necessary condition for  the C*-algebra of an amenable \'etale groupoid to be RFD.
\end{abstract}


\maketitle
%\tableofcontents
\section{Introduction}
A C*-algebra is said to be {\it residually finite-dimensional} (RFD) if its finite-dimensional representations separate its points. This natural `finiteness' property can be viewed as a C*-analogue of the \emph{maximal almost periodicity} (MAP) of groups, the property which requires that finite-dimensional representations separate points of a group. This analogy however has to  be taken with a grain of salt, as for example $SL_3(\Z)$ is a MAP group and yet $C^*(SL_3(\Z))$ is not RFD, as follows from the main result of \cite{BekkaInv} (and if we wanted to work with reduced group $C^*$-algebras, already $F_2$ is MAP and $C^*_r(F_2)$ is not RFD). The perfect symmetry is however recovered when one sticks to discrete amenable groups: as shown in \cite{BekkaLouvet}, following \cite{BekkaForum}, the group C*-algebra of an amenable group $G$ is RFD if and only if $G$ is MAP.
In general the question of what groups have RFD C*-algebras is old and challenging (see \cite[Chapter 9]{LZ} dedicated to this problem). Moreover
the property of being RFD lies in the heart of several long-standing problems in operator algebras. E.g.\ Kirchberg's conjecture or,  equivalently, the Connes Embedding Problem, states that $C^*(F_2\times F_2)$ is RFD (see \cite{Taka}). On purely C*-theoretic side,  the RFD property is central in the developments on the UCT conjecture \cite{Dadarlat1} and the problem of characterising subalgebras of AF-algebras \cite{Dadarlat2}, \cite{BO}.
 Recently RFD C*-algebras have also started to play a prominent role in the study of non-self-adjoint operator algebras (\cite{CDO}, \cite{Hartz}) and of decidability algorithms  \cite{FNT}, as well as in the purely algebraic questions \cite{ANT}.  Behavior of RFD under various constructions has  therefore always been a topic of considerable attention (a list of numerous results on this topic  can be found e.g.\ in \cite{Tatiana} and references therein).

This paper investigates the RFD property for two prominent classes of C*-algebras - crossed products by amenable actions of discrete groups  and C*-algebras of amenable \'etale groupoids.

Crossed products by amenable actions are among the most notable and well studied classes of C*-algebras.   In particular there are examples of RFD crossed products, but the most general result that has been obtained up to now is a result by Tomiyama  \cite{TomiyamaBook} that characterizes the RFD property of crossed products $C(X)\rtimes G$ of commutative C*-algebras    under the assumption that all irreducible representations of $G$ are finite-dimensional (in other words, by \cite{Thoma}, $G$ is virtually abelian).  In this paper we obtain a complete characterisation, for an arbitrary, not necessarily commutative, C*-algebra and any amenable action.  Below the notation $S(A)/_\approx$ and $\Rep(A)/_\approx$ is used for the state space and the collection of  all representations up to unitary equivalence.

\begin{thmx} \label{thm:A} Let $G$ be a discrete  group, $A$ a C*-algebra and $\alpha$ an amenable action of $G$ on $A$. Then the following conditions are equivalent:
	\begin{itemize}
		\item [(i)] $A\rtimes G$ is RFD;
				\item[(ii)] G is MAP and finite-dimensional elements of $\Prim(A)$ whose orbits are finite are dense in Prim(A).
		\item[(iii)] $G$ is MAP and every pure state of $A$ can be approximated in weak$^*$-topology by finite-dimensional  states whose orbits in $S(A)/_\approx$ are finite;
\item[(iv)] $G$ is MAP and any representation of $A$ is a corner of a representation which can be approximated in the SOT-topology by finite-dimensional representations whose orbits in $\Rep(A)/_\approx$ are finite.
		%\item[(v)] $G$ is MAP and the action $\alpha$  is approximately fd-periodic.
		% \item[(v)] $G$ is MAP and any infinite-dimensional irreducible representation of $A$ %on a separable Hilbert space
		%can be approximated in the SOT-topology by finite-dimensional irreducible  representations whose orbits in $\Rep(A)/_\approx$ are finite.
	\end{itemize}
\end{thmx}
%{\color{blue} In (iii) and (iv) is it better to replace finite by trivial or keep finite?}

In particular in the case when $A$ is commutative and $G$ is virtually abelian we obtain Tomiyama's theorem, and in the case of trivial $A$ we recover Bekka's theorem. Note that in (iii) and (iv) above we can also replace `finite' by `trivial'.

The equivalence $(i)\Leftrightarrow (ii)$ is inspired by Bekka's theorem and his idea of showing RFD by approximating the $\delta$-trace by finite-dimensional states. We note though that finding a right way of implementing this idea for crossed products finds significant obstacles when the C*-algebra is non-commutative. While finite-dimensional representations of $C(X)\rtimes G$ are completely described by Tomiyama, his construction simply does not work in the non-commutative case. We had to involve projective representations to be able to adjust Tomiyama's construction for general crossed products. The collection of representations we construct turn out to be sufficiently rich to allow the analysis of  the RFD property of crossed products by amenable actions.

The equivalences $(i)\Leftrightarrow (iii) \Leftrightarrow (iv)$ can be considered as a dynamical version of Exel-Loring characterization of RFD C*-algebras \cite{ExelLoring} (with a proof using different methods than \cite{ExelLoring}).  As Exel and Loring show, the possibility to describe RFD as a formally different approximation property for representations is very useful in certain situations, e.g. when dealing with free products. The equivalences above give such possibility for crossed products (see Subsection \ref{subsec:free} for an application of this).

Moreover as a byproduct we establish a new, even in the absence of a group action,  characterisation of the RFD property in the spirit of \cite{ExelLoring}: a $C^*$-algebra $A$ is RFD if and only if any irreducible representation of $A$ can be approximated pointwise in the SOT-topology by finite-dimensional irreducible  representations (Corollary \ref{RFDnew}).

As a particular case of the theorem above we obtain a characterization of RFD property for C*-algebras of semidirect products of groups by amenable groups. What is more, our methods apply also directly in the purely group-theoretic context of semidirect products of groups, and lead to characterisations of the MAP property,  the \emph{RF (residual finiteness) property} and  \emph{property FD} of Lubotzky and Shalom (\cite{LubotzkyShalom}). We mention a few of the examples  obtained from  the resulting statements, referring for further  results of this type to Sections \ref{sec:Semidirect} and \ref{sec:Examples}.

\begin{corx} \label{corB}
	
	Let $H$ be a finitely generated group with the property FD (for example, the free group $F_n$ or a surface group or any finitely generated amenable RF group). Then
 \begin{itemize}
\item [(i)] for any amenable RF group $G$, $H\rtimes G$ has FD;

\item [(ii)] for any  amenable MAP group $G$, the C*-algebra
$C^*(H\rtimes G)$ is RFD;

\item [(iii)] for any MAP group $G$, the group $H\rtimes G$ is MAP.
\end{itemize}

\medskip

On the other hand, let $H$ be a group with  property (T). Then
\begin{itemize}
	\item [(iv)]  if $H$ is MAP, then for any MAP group $G$,  $H\rtimes G$ is MAP;
	
	\item [(v)] if $C^*(H)$  is RFD, then for any MAP amenable group $G$,  $C^*(H\rtimes G)$ is RFD.
\end{itemize}
\end{corx}

In \cite[Th. 2.8]{LubotzkyShalom} Lubotzky and Shalom proved that semidirect products of $F_n$ and of surface groups with $\Z$ have the FD, but nothing was known regarding the question whether the FD is closed under semidirect products  with $\Z$ in general. In  the statement (i) above we proved that the FD is stable under   semidirect products  not only with $\Z$ but with any amenable RF group. Applying our results together with a recent resolution of the Baumslag's conjecture in \cite{KL} we can further deduce that $C^*(G)$ is RFD for any $G$ which is a 1-relator group with torsion.

One more  consequence  of the general results of our paper concerns wreath products. There is a well-known characterization of the RF property of wreath products  obtained by Gruenberg \cite{Grunberg}.
Here we obtain a characterization of the MAP property of wreath products:  $H \wr G$ is MAP if and only if either $H$ is MAP and $G$ is finite, or $H$ is trivial and $G$ is MAP,   or $H$ is abelian and $G$ is RF. We would like to note that the latter result, and  the characterizations of MAP and RF properties mentioned above  were also very recently obtained by Bekka by different methods in \cite{BekkaPreprint}.% and that we obtained these results independently from him.

\medskip

We now turn to discussing the RFD property for C*-algebras of amenable \'etale groupoids.
The latter algebras have recently taken a very central role in the classification programme for nuclear C*-algebras and the study of the UCT problem (\cite{Re}, \cite{BaL}). First we give  full description of their finite-dimensional representations (Theorem \ref{th:finitegpd}) and then use it to investigate when C*-algebras of amenable \'etale groupoids are RFD.  It turns out that the sufficient condition for the RFD property of the groupoid $C^*$-algebra is that the isotropy groups are MAP, and the action on the unit space has a dense set of finite orbits; we formulate it in the next theorem, and refer to Section \ref{Sec:grpd} for the detailed definitions.

\begin{thmx} \label{thm:C}
	 Let $\GG$ be an amenable \'etale groupoid with the locally compact unit space $X$.
	If  $C^*(\GG)$ is RFD then $X$ admits a dense set of periodic points; and if $X$ admits a dense set of periodic points with MAP isotropy subgroups, then $C^*(\GG)$ is RFD.
\end{thmx}

The latter result leads almost immediately to a characterisation of the RFD property for the crossed products arising from amenable actions of discrete groups on locally compact spaces.  This again generalizes Tomiyama's result and Bekka's theorem.

\medskip

The detailed plan of the paper is as follows: after this introduction, in Section \ref{Sec:grpd} we study the RFD property of (amenable) groupoid $C^*$-algebras. We provide a full description of their finite-dimensional representations in Theorem \ref{th:finitegpd} and prove Theorem C above (Theorem \ref{theorem:main}).  We also provide in Proposition \ref{prop:cntrex} an example of an amenable \'etale groupoid with a non-MAP isotropy subgroup and with an RFD C*-algebra, and ask whether the second implication in Theorem C can be reversed.  Section \ref{Sec:crossedprod} is devoted to the study of RFD property of C*-algebraic crossed products. We begin by discussing several equivalent formulations of `approximating representations of a $G$-C*-algebra by finite-dimensional representations with finite orbits' in Theorem \ref{thm:justact}, Remark 	\ref{rem:triv}  and Theorem \ref{thm:justactseparable}. We then proceed to construct and analyse finite-dimensional representations of the relevant crossed products and prove Theorem \ref{thm:A} (Theorem \ref{th:maincrossed}). The next subsection is devoted to discussing several consequences of that result, and we end Section \ref{Sec:crossedprod} by showing that in the context we consider the amenability of the action in fact implies the amenability of the acting group (Corollary \ref{cor:amen}). A short Section \ref{sec:Semidirect} is devoted to abstract characterisations of the MAP, RF, RFD and FD property of semidirect products. Finally in Section \ref{sec:Examples} we discuss several concrete examples of groups fitting in the context of our study.  In particular Corollary \ref{corB} is a combination of Theorem \ref{thm:FDsemid}, Theorem \ref{FD} and Corollary \ref{cor:T}.

\section{RFD Property for groupoid $C^*$-algebras}\label{Sec:grpd}

In this section we discuss the RFD property for $C^*$-algebras of amenable \'etale groupoids. All the topological groupoids considered in this paper will be locally compact \'etale Hausdorff second-countable; for brevity we will just say -- as above -- that $\GG$ is \'etale (and tacitly assume Hausdorffness, local compactness, and second-countability).

\medskip

Let $\GG$ be a groupoid with a locally compact unit space $\GG^{(0)}:=X$. Recall that for $x,y \in \GG^{(0)}$ we write $\GG_x^y:=\{\gamma \in \GG: s(\gamma)=x, r(\gamma) = y\}$, and call  $\GG_x^x$ the \emph{isotropy subgroup of $\GG$ at $x$}. Introduce the following  relation on $\GG^{(0)}$: for $x,y \in \GG^{(0)}$ we have $x\approx y$ if $\GG_x^y\neq \emptyset$. It is easy to check that $\approx$ is an equivalence relation; we shall call its equivalence classes \emph{orbits} of $\GG$ (inside $\GG^{(0)}$). A point $x\in \GG^{(0)}$ is said to be \emph{periodic} if its orbit is finite.

Note the following easy fact: if $x,y \in \GG^{(0)}$, $x\approx y$, then the isotropy subgroups
$\GG_x^x$ and $\GG_y^y$ are isomorphic: more specifically any $\gamma \in  \GG_x^y$ determines an isomorphism
\[ \GG_x^x \ni g \mapsto  \gamma g \gamma^{-1}\in \GG_y^y.\]

\subsection*{Finite-dimensional representations of groupoid $C^*$-algebras}

From now on we will assume that we work with a fixed  \'etale groupoid $\GG$  and use the notation and terminology of \cite{Aidan}.

Let us look first at finite-dimensional %(and often irreducible)
representations of $C^*(\GG)$.

\begin{lemma} \label{lem:finiteorbit}
	Suppose  that $\pi:C^*(\GG)\to B(H)$ is an irreducible finite-dimensional representation. Then there exists a finite orbit $F\subset  X$ and Hilbert spaces $(H_x)_{x \in F}$ such that $H \cong \bigoplus_{x\in F} H_x$ and $\pi|_{C_0(X)} = \bigoplus_{x \in F} \textup{ev}_x 1_{H_x}$.
\end{lemma}
\begin{proof}
It is clear that a restriction of $\pi$ to $C_0(X)$ (which is a $C^*$-subalgebra of $C^*(\GG)$, see for example \cite[Proposition 4.2.6 and Corollary 9.3.4]{Aidan})
decomposes as above for a certain finite set $F\subset X$. As we assumed irreducibility, it suffices to show that if $x, y \in F$ and $x\not \approx y$ 	then $P_x \pi(f) P_y = 0$ for each $f \in C_c(\GG)$, where we write $P_x:=1_{H_x}$, $P_y:=1_{H_y}$.

It suffices to prove the above fact for $f\in C_c(\GG)$ supported on an open bisection $U$. If $x \notin r(U)$ then we can find a function $g\in C_c(X)\subset C_c(\GG)$ such that $g|_{r(\supp(f))}=1$ and $g(x)=0$ (we are separating a compact set from a point). But then $g \star f = f$, and $P_x \pi(g) =0$, so that we have $P_x \pi(f) P_y = P_x \pi(g) \pi(f) P_y = 0$.

Assume then that $x \in r(U)$; let $\gamma\in U$ be such that $r(\gamma)=x$. By our assumption ($\G_y^x=\emptyset$) we have $s(\gamma)\neq y$. Let then $V_1$, $V_2$ be disjoint open sets in $X$ such that $y\in V_1$, $x\in V_2$.  Let then $V_x$ be an open neighbourhood of $x$ contained in $r(s^{-1}(V_2) \cap U)$, and $V_y$ an open neighbourhood of $y$ contained in $V_1$ such that in addition $V_x \cap F = \{x\}$, $V_y \cap F =\{y\}$. Consider functions $g,h\in C_c(\GG)$ with $\supp(g)\subset V_x$ and $\supp(h)\subset V_y$, $g(x)=h(y)=1$. Then we have on one hand
$P_x \pi(g)=P_x$,  $\pi(h) P_y =P_y$ and on the other hand $g\star f \star h =0$. Let us clarify the last statement: suppose that $g\star f \star h(\gamma) \neq 0$ for some $\gamma \in U$. Then we have $r(\gamma) \in V_x$ and $s(\gamma) \in V_y$; the first condition implies  that $s(\gamma)\in V_2$ and we arrive at a contradiction.
Thus finally we have $P_x \pi(f) P_y = P_x \pi(g) \pi(f) \pi(h) P_y = 0$.
\end{proof}	

We will continue with two lemmas, the first of which is essentially a consequence of the proof above.

\begin{lemma} \label{lem:disjoint}
	Suppose  that $\pi:C^*(\GG)\to B(H)$ is a representation for which there exists a finite orbit $F\subset  X$ such that $H \cong \bigoplus_{x\in F} H_x$ and $\pi|_{C_0(X)} = \bigoplus_{x \in F} \textup{ev}_x 1_{H_x}$. Let $x,y \in F$, and let $f \in C_c(\GG)$ be supported on an open bisection $U$. If $f|_{\GG_x^y}=0$ then $1_{H_y} \pi(f) 1_{H_x} =0$.
\end{lemma}

\begin{proof}
If $x \neq y$, one can simply repeat the second part of the proof of the previous lemma. Consider then the case where $y=x$. As in the previous lemma one can assume that $x \in r(U)$. Let then $\gamma_0 \in U$ be such that $r(\gamma_0)=x$.
If $s(\gamma_0)\neq x$, we can again proceed as in the previous lemma. On the other hand if $s(\gamma_0)=x$, we have $\gamma_0 \in \G_x^x$ so that $f(\gamma_0)=0$; and as $f$ is continuous, for every $\epsilon>0$ one can find a neighbourhood $V$ of $\gamma_0$ such that $|f(\gamma)|< \epsilon$ for every $\gamma \in V$.  We can then find functions $g,h \in C_c(X)$ bounded by $1$, supported inside respectively $r^{-1}(V)$ and $s^{-1}(V)$, such that $g(x)=h(x)=1$ and
$P_x \pi(g) = P_x$, $\pi(h) P_x = P_x$. It is easy to check  that $|g \star f \star h|$ is bounded by $\epsilon$. But as for functions supported on bisections the norm in  $C^*(\GG)$ is bounded by the uniform norm (see for example \cite[Proposition 9.2.1]{Aidan}), and representations are contractive,
we obtain
\[ \| P_x \pi(f) P_x\| = \|P_x \pi(g) \pi(f) \pi(h) P_x \| \leq \|g \star f \star h\|_{\infty} \leq \epsilon,\]
which ends the proof.
	\end{proof}

\begin{lemma} \label{lem:unique}
	Suppose  that $\pi:C^*(\GG)\to B(H)$ is a representation for which there exists a finite orbit $F\subset  X$ such that $H \cong \bigoplus_{x\in F} H_x$ and $\pi|_{C(X)} = \bigoplus_{x \in F} \textup{ev}_x 1_{H_x}$. Let $x,y \in F$, and let $\gamma \in \GG_{x}^y$. If $f_1, f_2 \in C_c(\GG)$ are functions supported on open bisections (say $U_1$ and $U_2$) such that $f_1(\gamma)=f_2(\gamma)=1$, then $1_{H_y} (\pi(f_1)-\pi(f_2)) 1_{H_x} =0$.
\end{lemma}

\begin{proof}
The assumptions of the lemma guarantee that $\gamma \in U_1 \cap U_2$. Find an open neighbourhood $W$ of $\gamma$ contained in $U_1 \cap U_2$, and then open sets $V_1$  and $V_2$ in $X$ such that $V_1 \cap F= \{y\}$, $V_2 \cap F=\{x\}$ and $V_1 \subset r(W)$. For $i=1,2$ let $g_i \in C_c(X)\subset C_c(\GG)$ be supported on $V_i$ and such that $g_1(x)=g_2(y)=1$, so that we have $\pi(g_1) = 1_{H_y}$, $\pi(g_2) = 1_{H_x}$.
Consider the functions $k:=g_1 \star f_1 \star g_2, l:=g_1 \star f_2 \star g_2$, both  in $C_c(\GG)$. If $\zeta\in \GG$ and $k(\zeta) \neq0$, then $\zeta \in U_1$ and $r(\zeta)\in V_1$, so that $\zeta \in W \subset U_1 \cap U_2$. Similarly if $l(\zeta)\neq 0$ then $\zeta \in U_1 \cap U_2$. This implies that $k-l$ is a function supported on an open bisection, and we can apply the last lemma to see that  $1_{H_y} \pi(k-l) 1_{H_x} =0$.
But   $1_{H_y} \pi(k-l) 1_{H_x} = 1_{H_y} \pi(g_1) \pi(f_1-f_2) \pi(g_2) 1_{H_x} =
1_{H_y}  \pi(f_1-f_2)  1_{H_x}$, which ends the proof.
\end{proof}

\begin{lemma}\label{lem:concreteform}
Suppose that $F\subset X$ is a finite orbit of the form $\{x_1, \ldots, x_n\}$, where $n = |F|$. For all $i,j=1, \ldots, n$ write $\GG_{j}^i:= \GG_{x_j}^{x_{i}}$, set $\gamma_{1}=e_{\GG_1^1}$   and for $i=2,\ldots, n$ fix  an element $\gamma_i \in \GG_{1}^{i}$, writing $\overrightarrow{\gamma}=(\gamma_1, \ldots, \gamma_n)$. Consider an arbitrary representation $\rho: \GG_{1}^1 \to B(\Hil_1)$, where $\Hil_1$ is a Hilbert space. Further assume we have a collection of unitaries ${\bf{z}}:=(z_i)_{i=1}^{n}$ such that $z_1=1_{H_i}$ and $z_i:\Hil_1\to \Hil_{i}$ for $i=2,\ldots, n$, where $\Hil_2=\cdots=\Hil_n = \Hil_1$. Finally
introduce a new Hilbert space $\mathcal{H}:=\bigoplus_{i=1}^n \Hil_i$. Then the `block-matrix' formula
\[ (\pi_{F, \rho, \bf{z}, \overrightarrow{\gamma}} (f))_{i,j} = \sum_{g \in \GG_j^i} f(g) z_{i} \rho (\gamma_{i}^{-1}g \gamma_{j}) z_j^*\in B(\mathcal{H}),  \]
$f \in C_c(\GG)$, $i,j=1, \ldots, n$,
defines a representation of $C^*(\GG)$ on $B(\mathcal{H})$.
	\end{lemma}

\begin{proof}
The proof is a simple check. Choose $(F, \rho, \bf{z}, \overrightarrow{\gamma})$ as above and  write $\pi:=\pi_{F, \rho, \bf{z}, \overrightarrow{\gamma}}$.
	
First note that each of the sums in the displayed formula is in fact finite, as $\GG$ is \'etale, so each $\GG_j^i$ is discrete (and $f$ is compactly supported). The correspondence $f\to \pi(f)$ is clearly linear.

Fix $i,j=1, \ldots, n$. We have  $(\pi(1_X))_{i,j}= \delta_{i,j} z_i \rho(\gamma_i^{-1}\gamma_i) z_i^* = \delta_{i,j} 1_{H_i}$. If $f \in C_c(\GG)$ then
\begin{align*}  (\pi (f^*))_{i,j} &= \sum_{g \in \GG_j^i} f^*(g) z_{i} \rho (\gamma_{i}^{-1}g \gamma_{j}) z_j^*
= \sum_{g \in \GG_j^i} \overline{f(g^{-1})} z_{i} \rho (\gamma_{i}^{-1}g \gamma_{j}) z_j^*
\\& = \sum_{g \in \GG_i^j} \overline{f(g)} z_{i} \rho (\gamma_{i}^{-1}g^{-1} \gamma_{j}) z_j^* =
 \sum_{g \in \GG_i^j} \overline{f(g)} z_{i} \rho ((\gamma_{j}^{-1}g \gamma_{i})^{-1}) z_j^*
\\&=
\sum_{g \in \GG_i^j} \overline{f(g)} z_{i} \rho ((\gamma_{j}^{-1}g \gamma_{i}))^* z_j^*
=
 \left( \sum_{g \in \GG_i^j} f(g) z_{j} \rho ((\gamma_{j}^{-1}g \gamma_{i})) z_i^* \right)^*
\\&= \left((\pi (f))_{j,i}\right)^*, \end{align*}
so that $\pi (f^*) = \left(\pi (f)\right)^*$.
Finally if $f, h\in C_c(\GG)$ and $g \in \GG_j^i$ then we have, setting $\GG^i=\{g \in \GG: r(g) =x_i\}$,
\[ (f\star h)(g) = \sum_{\kappa \in \GG^i} f(\kappa)h(\kappa^{-1}g) = \sum_{k=1}^n
\sum_{\kappa \in \GG_k ^i} f(\kappa)h(\kappa^{-1}g),\]
as if $r(\kappa)= x_i$ then $s(\kappa) \in F$. Thus
\begin{align*} (\pi (f\star h))_{i,j} &= \sum_{g \in G_j^i} (f \star h)(g) z_{i} \rho (\gamma_{i}^{-1}g \gamma_{j}) z_j^*
= \sum_{g \in G_j^i} \sum_{k=1}^n
\sum_{\kappa \in \GG_k^i} f(\kappa)h(\kappa^{-1}g) z_{i} \rho (\gamma_{i}^{-1}g \gamma_{j}) z_j^*.
\end{align*}
On the other hand
\begin{align*} (\pi (f)  \pi (h))_{i,j}
	&= \sum_{k=1}^n 	(\pi_{F, \rho, \bf{z}} (f))_{i,k}  (\pi_{F, \rho, \bf{z}} (h))_{k,j}
	\\&= \sum_{k=1}^n \sum_{g \in \GG_k^i} f(g) z_{i} \rho (\gamma_{i}^{-1}g \gamma_{k}) z_k^*
	\sum_{g'\in \GG_j^k} h(g') z_{k} \rho (\gamma_{k}^{-1}g' \gamma_{i}) z_j^*
	\\&=
\sum_{k=1}^n \sum_{g \in \GG_k^i} 	\sum_{g'\in G_j^k} f(g) h(g')  z_i \rho (\gamma_{i}^{-1}g \gamma_{k}) \rho (\gamma_{k}^{-1}g' \gamma_{i}) z_j^*
\\&=
\sum_{k=1}^n \sum_{g \in \GG_k^i} 	\sum_{g'\in G_j^k} f(g) h(g')  z_i \rho (\gamma_{i}^{-1}g g' \gamma_{i}) z_j^*.
\end{align*}	
A straightforward reindexing argument shows that the two  sums displayed above coincide.
\end{proof}

The result below could be deduced from the disintegration theorem, but a direct proof presented here avoids any measure-theoretic considerations (as expected due to finite-dimensionality).

\begin{theorem}\label{th:finitegpd}
	Suppose that $\GG$ is an  \'etale groupoid.
	Every irreducible finite dimensional representation of $C^*(\GG)$ arises via the construction in Lemma \ref{lem:concreteform} (naturally with $\Hil_1$ finite-dimensional).
\end{theorem}

\begin{proof}
	Consider an irreducible representation $\pi:C^*(\GG) \to B(\mathcal{H})$, where $\mathcal{H}$ is a finite-dimensional Hilbert space.  Lemma \ref{lem:finiteorbit} provides us with a finite orbit $F\subset X$ of the form $\{x_1, \ldots, x_n\}$, with  $\mathcal{H} \cong \bigoplus_{i=1}^n \Hil_i$ and $\pi|_{C_0(X)} = \bigoplus_{i=1}^n \textup{ev}_{x_i} 1_{\Hil_i}$. For $i,j=1,\ldots, n$ we will write $P_i:=1_{\Hil_i}$  and $\GG_{j}^i:= \GG_{x_j}^{x_{i}}$.
	
%Fix for the rest of the proof disjoint open sets $V_i\subset X$ such that $V_i\cap F=\{x_i\}$ for each $i=1,\ldots,n$.

Take any $\gamma\in G_1^1$. Consider a function $f\in C_c(\GG)$ supported on an open bisection $U$ and such that $f(\gamma)=1$ and set $\rho_1(\gamma)= P_1 \pi(f) P_1$. Lemma  \ref{lem:unique} implies that $\rho_1(\gamma)$ -- which we will view as an operator on $\Hil_1$ -- does not depend on the choice of $f$ as above.  Lemma \ref{lem:disjoint} implies then that for any $f \in C_c(\GG)$ supported on an open bisection we have
\[ P_1\pi(f) P_1 = \sum_{\gamma \in \GG_1^1} f(\gamma) \rho_1(\gamma);\]
and as the above formula is clearly linear we can immediately deduce that it remains valid for arbitrary $f \in C_c(G)$.

Note that Lemma \ref{lem:disjoint} implies that if  $f\in C_c(\GG)$ is supported on an open bisection $U$ such that $s(U) \cap F = \{x_i\}$ and $r(U) \cap F = \{x_j\}$ (for some $i,j\in \{1, \ldots,n\}$) then $\pi(f) = P_j \pi(f) P_i$. Thus given $\lambda_1, \lambda_2 \in \GG^1_1$ and choosing suitable functions $f_1,f_2 \in C_c(\GG)$ supported on sufficiently small open bisections and such that $f_1(\lambda_1)=1$, $f_2 (\lambda_2)=1$ we may consider $f_1 \star f_2$ and note that we have
\[ \rho_1(\lambda_1 \lambda_2) = P_1 \pi(f_1 \star f_2) P_1 = P_1 \pi(f_1) \pi(f_2) P_1 =
P_1\pi(f_1) P_1\pi(f_2) P_1 = \rho_1(\lambda_1) \rho_1(\lambda_2).\]
This implies that $\rho_1:\GG_1^1 \to B(\Hil_1)$ is  a representation. Naturally we can repeat this argument and construct for each $i=1,\ldots,n$ representations $\rho_i:\GG_i^i \to B(\Hil_i)$ such that for any $f \in C_c(\GG)$ we have
\[ P_i\pi(f) P_i = \sum_{\gamma \in G_i^i} f(\gamma) \rho_i(\gamma).\]
Consider now $i\in\{2,\ldots,n\}$, choose some $\gamma_i \in \GG_{1}^{i}$  and suppose that $f\in C_c(\GG)$ is supported on an open bisection $U$ and such that $f(\gamma_i)=1$. Set $z_i= P_i \pi(f) P_1$. Lemma  \ref{lem:unique} implies that $z_i$ does not depend on the choice of $f$ as above. The same argument can be applied to $\gamma_i \gamma$ for a fixed $\gamma \in \GG_1^1$, allowing us further to deduce that for any $f\in C_c(\GG)$ supported on an open bisection $U$ and such that $f(\gamma_i\gamma)=1$ we must have $P_i \pi(f) P_1 = z_i \rho_1(\gamma)$. This in turn implies -- again first arguing for functions supported on bisections -- that in fact for $f \in C_c(\GG)$ we have
\[ P_i\pi(f)P_1 = \sum_{\gamma\in \GG_1^1} f(\gamma_i \gamma) z_i \rho_1(\gamma)= \sum_{\kappa\in \GG_1^i} f(\kappa) z_i \rho_1(\gamma_{i}^{-1} \kappa) .\]
It remains to note that using analogous reasoning we can deduce that in fact $z_i$ is a partial isometry with the source space $\Hil_1$ and the range space $\Hil_i$, and further the representations $\rho_1$ and $\rho_i$ are naturally connected via the isomorphism relating $\GG_1^1$ and $\GG_i^i$. By the last statement we mean that for every $\kappa \in \GG_i^i$ we have
\[\rho_i(\kappa) =z_i^* \rho_1(\gamma_i^{-1} \kappa \gamma_i) z_i.\]
We leave the rest of the argument (which amounts to setting $\rho:=\rho_1$, identifying the Hilbert spaces $\Hil_i$ with each other and somewhat tediously collecting the pieces) to the reader.
\end{proof}

\subsection{RFD property for groupoid $C^*$-algebras -- positive results}

In this subsection we will provide some sufficient and necessary results for the group C*-algebra of an amenable \'etale groupoid to be RFD. Note that by \cite[Proposition 6.1.8]{ADR} if $\GG$ is (topologically) amenable then the reduced and universal groupoid C*-algebras of $\GG$ are canonically isomorphic.

Before we formulate the next lemma, recall that $\Phi:C^*_r(\GG)\to C_0(X)$ denotes the standard faithful conditional expectation, which for $f\in C_c(\GG)$ is simply given by evaluation:
$\Phi(f) = f|_{X}$ (see \cite[Proposition 4.2.6]{Aidan}).

\begin{lemma}\label{mainrevisited} Let $\GG$ be an amenable \'etale groupoid, and let $x\in X$ be a point with a finite orbit  $F = \{x_1, \ldots,x_n\}$. Assume that $\GG_{x_1}^{x_1}$ is MAP.   Then for any $\epsilon>0$ and $f \in C_c(\GG)$ there exists a finite-dimensional representation $\pi$ of $C^*(\GG)$ such that $$\left| \frac{\Phi(f)(x_1)+\ldots +\Phi(f)(x_n)}{n} - \tr \, \pi(f)\right|<\epsilon.$$
\end{lemma}
\begin{proof} Fix $f \in  C_c(\GG)$ and $\epsilon>0$. Set as usual $\GG_i^j:= \GG_{x_i}^{x_j}$ for all $i,j=1\ldots, n$ and note first that the restrictions $f|_{\GG_i^i}$ are finitely supported. Set further $\gamma_1:=e_{\GG_1^1}$, fix for each $i=2,\ldots,n$ an element $\gamma_i\in \GG_1^i $ and put $E=\{g \in \GG_1^1: \exists_{i\in \{1,\ldots,n\}} f(\gamma_ig\gamma_i^{-1}) \neq 0\}$. The set $E$ is finite, so we can further put $M:=\max_{\gamma \in E, i=1,\ldots,n} |f(\gamma_i\gamma\gamma^{-1}_i)|$.

As $\GG_1^1$ is assumed to be MAP
there is a finite-dimensional representation $\tilde{\rho}$ of $\GG_1^1$ such that
	$\tilde{\rho}(h) \neq 1$ for any $e\neq h\in E$. Let $\chi$ be the trivial representation of $\GG_1^1$. Then, denoting by $\tr$ the normalised matrix trace, we have
	$$|\tr (\tilde{\rho}\oplus \chi)(h)|\neq 1,$$ for any $e\neq h\in E$. Taking an appropriate tensor power of $\tilde{\rho}\oplus \chi$ we obtain a finite-dimensional representation $\rho: \GG_1^1\to B(\Hil_1)$ such that $$|\tr  \rho(h)|<\frac{\epsilon}{\sharp(E)M},$$ for any $e\neq h\in E$.

	Using $\rho$, $F$, a trivial collection $\bf{z}=(\id_{\Hil_1}, \ldots,\id_{\Hil_1})$  and $\overrightarrow{\gamma}=(\gamma_1, \ldots, \gamma_n)$ we construct a finite-dimensional representation $\pi:=\pi_{F, \rho, \bf{z}, \overrightarrow{\gamma}}$ of $C^*(\GG)$ as in Lemma \ref{lem:concreteform}. The explicit form of $\pi$ implies that we have	
\begin{align*}\tr (\pi(f)) &= \frac{1}{n}\sum_{i=1}^n \tr (\pi(f)_{i,i}) =
\frac{1}{n} \sum_{i=1}^n  \sum_{g \in \GG_i^i} f(g) \tr (\rho (\gamma_{i}^{-1}g \gamma_{i}))
\\&= \frac{1}{n} \left( \sum_{i=1}^n f(x_i) + \sum_{i=1}^n  \sum_{g \in \GG_i^i\setminus\{x_i\}, f(g)\neq 0}  f(g) \tr (\rho (\gamma_{i}^{-1}g \gamma_{i})) \right).
\end{align*}
It thus follows that
\begin{align*}
\left|\tr (\pi(f)) - \frac{\Phi(f)(x_1)+\ldots +\Phi(f)(x_n)}{n} \right| &\leq \frac{1}{n}
 \sum_{i=1}^n  \sum_{g \in \GG_1^1, g \neq e, f(\gamma_ig\gamma_i^{-1})\neq 0}  \left|f(\gamma_i g \gamma_i^{-1}) \tr (\rho (g))\right| \\&\leq \sum_{g \in E} M \tr (\rho(g)<  M \frac{\epsilon}{M}= \epsilon.
\end{align*}	

\end{proof}

The following is the main result of this Section.

\begin{theorem}\label{theorem:main}
	Let $\GG$ be an amenable \'etale groupoid.
	If  $C^*(\GG)$ is RFD then $X$ admits a dense set of periodic points; and if $X$ admits a dense set of periodic points with MAP isotropy subgroups, then $C^*(\GG)$ is RFD.
\end{theorem}
\begin{proof}
	The first statement is a consequence of Lemma \ref{lem:finiteorbit}. Indeed, if $V\subset X$ is a non-empty open set such that no element of $V$ has a finite orbit, then for any $f \in C_c(X)\subset C_c(\GG)$ supported on $V$ and for any finite-dimensional representation $\pi:C^*(\GG) \to B(\Hil)$ we have $\pi(f)=0$.
	
	Assume then that $X$ admits a dense set of periodic points with MAP isotropy subgroups.
%		if  $x\in X$ then the operation of restricting functions in $C_c(\GG)$ to $\GG_{x}^x$ defines a surjective unital $^*$-homomorphism $q:C^*(\GG) \to C^*(\GG_x^x)$. Thus the amenability of $\GG$ implies first that $C^*(\GG)$ is nuclear, and then so is its quotient $C^*(\GG_x^x)$. This implies that

\begin{comment}

Note that \cite[Proposition 5.1.1]{ADR} implies that $\GG_x^x$ is amenable (which we use below, applying Lemma \ref{mainrevisited}).

\end{comment}

	 Let $0\neq \tilde{f}$ be a positive element in $C^*(\GG)$. Since $\GG$ is amenable, the non-negative continuous function $\Phi(\tilde{f})$ is  not everywhere zero. Hence there is $x\in X$ with finite orbit $F=\{x_1, \ldots,x_n\}$ such that $\Phi(\tilde{f})(x_1)\neq 0$. Then
	$$C:= \frac{\Phi(\tilde{f})(x_1) +\ldots + \Phi(\tilde{f})(x_n)}{n}>0.$$
	We find $f\in \mathcal C_c(\GG)$ such that $\|f - \tilde{f}\|<C/4.$  By Lemma \ref{mainrevisited} there is a finite-dimensional representation $\pi$ of $C^*(\GG)$ such that
$$\left| \frac{\Phi(f)(x_1)+\ldots +\Phi(f)(x_n)}{n} - \tr\, \pi(f)\right|<\frac{C}{4}.$$
	Then \begin{multline*} \tr \, \pi(\tilde{f}) \ge \tr\, \pi(f) - \|\pi(f- \tilde{f})\| > \tr\, \pi(f) - C/4\\
	>\frac{\Phi(f)(x_1)+\ldots +\Phi(f)(x_n)}{n}  -2C/4\\ > \frac{\Phi(\tilde{f})(x_1)+\ldots +\Phi(\tilde{f})(x_n)}{n}  -3C/4= C/4 \neq 0.
	\end{multline*}
	Hence $\pi(\tilde{f})\neq 0$. Thus finite-dimensional representations of $C^*(\GG)$  separate positive elements and therefore $C^*(\GG)$  is RFD.
\end{proof}

It is not clear whether the second condition appearing in the theorem is in fact neccessary for $C^*(\GG)$ being RFD. We will return to this in the next Subsection: see especially Question \ref{quest}.

Let us now turn to corollaries. The notion of amenable action we use below is the one which appears in the compact case for example in  \cite[Section 4.4]{BO}
(and in \cite[Remark 9.17]{Dana} for actions on locally compact spaces). But first, a remark.

\begin{remark}\label{rem:mult}
	If $A$ is an RFD $C^*$-algebra, so is its multiplier C*-algebra $M(A)$.
\end{remark}
\begin{proof}
Suppose that $(\pi_i:A \to B(\Hil_i))_{i \in I}$ is a separating family of finite-dimensional representations for $A$. One can assume that each $\pi_i$ is non-degenerate. It is then easy to check that the canonical extensions $(\tilde{\pi}_i:M(A) \to B(\Hil_i))_{i \in I}$ separate points of $M(A)$.
\end{proof}

\begin{corollary}\label{TransformationGroupoid} Let $G$ be a discrete  group, $X$ a non-empty locally compact  Hausdorff   second-countable space and $\alpha$ an amenable action of $G$ on $X$. The following are equivalent:
	\begin{itemize}
		\item[(i)] $G$ is MAP and the union of finite orbits of the $G$-action is dense in $X$;
		\item[(ii)]  $C_0(X)\rtimes G$ is RFD.
	\end{itemize}
The conditions above  imply that $G$ itself is amenable.
\end{corollary}
\begin{proof}
	
(i)$\Longrightarrow$ (ii)	

Note that if  $x\in X$ belongs to a finite orbit, and $\GG$ is the groupoid associated to the action in question, then $\GG_x^x$ is the stabiliser of $x$ with respect to the $G$-action. This implies that $\GG_x^x$ is MAP and the implication follows from the second part of Theorem \ref{theorem:main}.

(ii)$\Longrightarrow$ (i)

Theorem \ref{theorem:main} implies that the action of $G$ on $X$ admits a dense set of periodic orbits. If an amenable action admits a finite orbit, say $\{x_1, \ldots,x_n\}$, it also admits an invariant probability measure $\sum_{i=1}^n \delta_{x_i}$, so $G$ itself must be amenable (\cite[Exercise 4.4.4]{BO}). Thus we have a natural inclusion $C^*_r (G)\subset M(C_0(X) \rtimes G)$, and $M(C_0(X) \rtimes G)$ must be RFD by Remark \ref{rem:mult}. But then $G$ is MAP and (i) holds.

The last statement follows from the argument above.

\end{proof}

We would like to note that the last result generalises \cite[Theorem 4.1.10]{TomiyamaBook} which proves the same result under the assumption that $G$ is virtually abelian (i.e.\ all irreducible representations of $G$ are finite-dimensional); this is of course  stronger than $G$ being MAP.
The corollary implies in particular that in some contexts the RFD property of the crossed product is generic; we will formulate it rigorously in a slightly more general context in Proposition \ref{prop:generic} below.

The last corollary implies in particular that if  $H, G$ are discrete groups, $H$ is abelian and $G$ amenable, then the semidrect product $H \rtimes G$ is MAP if and only if $G$ is MAP and the induced action of $G$ on $\hat{H}$ has a dense set of periodic points.
This was already quoted and applied in \cite{ES}. Later in Proposition \ref{MAPversion} we will see a generalisation of this fact.
%\begin{corollary} \label{cor:amenableonabelian}
%Suppose that $\Lambda, \Gamma$ are discrete groups, $\Lambda$ is abelian, $\Gamma$ amenable. A semidrect product $\Lambda \rtimes \Gamma$ is MAP if and only if $\Gamma$ is MAP and the induced action of $\Gamma$ on $\hat{\Lambda}$ has a dense set of periodic points.
%\end{corollary}

%\begin{proof}
%	As $\Lambda \rtimes \Gamma$ is amenable, it is MAP if and only if $C^*(\Lambda \rtimes \Gamma)$ is RFD. But $C^*(\Lambda \rtimes \Gamma) \approx C^*(\Lambda )\rtimes \Gamma \approx C(\hat{\Lambda}) \rtimes \Gamma$, so the result follows from Corollary \ref{TransformationGroupoid}.
%\end{proof}

\subsection{Example of RFD amenable groupoid with a non-MAP isotropy subgroup}

We begin with describing a general construction, in a sense resembling the HLS groupoids of
\cite{Rufus}.

\begin{proposition} \label{prop:indgrp}
Let $\Gamma$ be a countable amenable group with an increasing sequence of subgroups $(\Gamma_n)_{n=1}^\infty$ such that $\Gamma=\bigcup_{n\in \N} \Gamma_n$. Set $X =\N \cup \{\infty\}$, the one-point compactification of $\N$ at infinity, with the usual order. Let $\GG= \bigsqcup_{x \in X} \Gamma_x$ be the bundle of isotropy groups with $\Gamma_\infty=\Gamma$. This means that $\GG$ is a groupoid with the unit space $X$, with the source and range maps coinciding, and the product defined pointwise in fibers as the relevant group multiplication.
Equip $\GG$ with the topology given by the basis $\{U_{g,n}: n \in \N, g \in \Gamma_n\}$,
 with $U_{g,n}=\{(g,m): m \in X, m \geq n\}$.
 Then $\GG$ (which we will also denote $\GG(\Gamma, (\Gamma_n)_{n=1}^\infty)$ to signify the dependence on the choice of the relevant groups) is an amenable \'etale groupoid, and
 \[C^*(\G) \approx \left\{(a_n)_{n=1}^\infty \in \prod_{n=1}^\infty C^*(\Gamma_n): (\iota_n(a_n) )_{n=1}^\infty \textup{ converges in }C^*(\Gamma)\right\},\]
where $\iota_n:C^*(\Gamma_n) \to C^*(\Gamma)$ denote the canonical inclusions.
\end{proposition}
\begin{proof}
It is easy to check that the proposed basis of neighbourhoods defines a locally compact topology on $\GG$ (note that the sets $U_{g,n}$ are compact). Remark also that if $F\subset \GG$ is compact then $\{g \in \Gamma: (g,m)\in F \textup{ for some }m \in X\}$ is finite (otherwise we could cover $F$ by infinitely many disjoint open sets). Each isotropy group is discrete in the relative topology, and continuity of groupoid operations is clear. Thus $\GG$ becomes an \'etale groupoid. It is amenable by \cite[Theorem 3.5]{Renault}.

In the following it will be slightly easier to work with the reduced picture. Given $f \in C_c(\GG)$ define $f_x:=f|_{\GG_x}$.
The reduced representation $\pi:C_c(\GG)\to B(\Hil)$ is a direct sum of representations indexed by $x \in X$, more specifically for $f \in C_c(\GG)$ we have
$\pi(f) = \oplus_{x \in X} \pi_x (f_x) \in \prod_{x \in X} C^*(\Gamma_x) \subset \prod_{x \in X} B(\ell^2(\Gamma_x))$, where $\pi_x$ denotes the left regular representation of $\Gamma_x$.  By the statement in the second sentence of the proof above we have that if we define $f_x:=f|_{\Gamma_x}$ then all functions $f_x$ are supported in the same finite subset $K\subset \Gamma$; moreover we have the pointwise convergence of $(f_n(g))_{n \in \N}$ to $f_\infty(g)$ for each $g \in K$. This implies that the sequence $(\iota_n(\pi_n(f_n)))_{n \in \N}$ converges  to $\pi_\infty(f_\infty)$ inside $C^*_r(\Gamma)$. As the maps $\iota_n$ are isometric we deduce that for each $f \in C_c(\G)$ we have
$\|\pi(f)\|_{B(\Hil)} = \sup_{n \in \N} \pi_n(f_n)$. Thus we have an injective embedding
$\alpha: C^*_r(\GG) \to \prod_{n\in \N} C^*(\Gamma_n)$.

By the argument above we have that for every $f \in C_c(\GG)$ the sequence $\iota_n(\alpha(f)_n)$ converges inside $C^*(\Gamma)$. The standard Cauchy argument -- and again the fact that $\iota_n$ are isometries -- implies that in fact $\iota_n(\alpha(f)_n)$ converges inside $C^*(\Gamma)$ for every $f \in C^*_r(\GG)$.

 Suppose that we have a sequence $(x_n)_{n=1}^\infty \in \prod_{n\in \N} C^*(\Gamma_n)$ such that $(\iota_n(x_n))_{n=1}^\infty$ converges to $z \in C^*(\Gamma)$. For any $\epsilon >0$ we have $z' \in \C[\Gamma]$ such that $\|z'-z\|_{C^*(\Gamma)}< \epsilon$. Let $P_n:C^*(\Gamma) \to C^*(\Gamma_n)$ denote the usual conditional expectation given by the (extension of) restriction of functions from $\Gamma$ to $\Gamma_n$. It is easy to see that the sequence $(\iota_n\circ P_n (z'))$ is eventually constant, and that if we view $P_n(z')$ as a function on $\Gamma_n$ then the sequence $((P_n z')_{n \in \N}, z')$ defines a function $f \in C_c(\GG)$. Let $N\in \N$ be such that $\|\iota_n(x_n) - z\| < \epsilon$ and $\iota_n\circ P_n (z') = z'$ for $n\geqslant N$. It is easy to see that the image of $\alpha$ contains all finite sequences in $\prod_{n\in \N}C^*(\Gamma_n)$, so we can find $h\in C^*_r(\GG)$ such that
\[ \alpha(h)_n = \begin{cases} x_n & \textup { if }  n <N, \\ 0 & \textup { if }  n \geqslant N.  \end{cases}\]
 Set then $g \in C_c(\GG)$ as follows:
 \[ g_n = \begin{cases} 0 & \textup { if }  n <N, \\ f_n &\textup { if }  n \geqslant N.  \end{cases}\]
We want to analyse the distance of $\alpha(g+h)$ from $(x_n)_{n=1}^\infty$. For $n <N$ we have
\[ \alpha(g+h)_n - x_n = 0\]
and for $n \geq N$
\begin{align*} \|\alpha(g+h)_n - x_n \|&=\|g_n - x_n\| = \|P_n z' - x_n \|  = \|\iota_n(P_n z') - \iota_n (x_n)\| \\&\leq \|\iota_n(P_n z') - z \|+\|z - \iota_n(x_n)\|
=\|z - z' \| +\|z - \iota_n(x_n)\| < 2 \epsilon.\end{align*}
Thus $\|\alpha(g+h) - (x_n)_{n=1}^\infty \| \leq 2 \epsilon$ and
as the image of $\alpha$ is closed we have that  $ (x_n)_{n=1}^\infty \in \alpha(C^*(\GG))$ and the proof is finished.
\end{proof}

We will now apply the construction in the context of groupoid C*-algebras with the RFD property.
\begin{proposition}\label{prop:cntrex}
Let $\Gamma$ be a countable non-MAP nilpotent group, e.g.\ $$\mathbb H_3(\mathbb Q) = \left\{  \left( \begin{array}{ccc}  1 & x & y \\ 0 & 1 & z\\ 0 & 0 & 1 \end{array}  \right) : x,y,z\in \mathbb Q \right\}$$  (see \cite[Section 4]{ES}). Consider any increasing sequence of finitely generated subgroups  $(\Gamma_n)_{n=1}^\infty$ such that $\Gamma=\bigcup_{n\in \N} \Gamma_n$.
Let $\GG:=\GG(\Gamma, (\Gamma_n)_{n=1}^\infty)$ be constructed as in Proposition \ref{prop:indgrp}. Then $\GG$ is an amenable \'etale groupoid with a non-MAP isotropy subgroup
and $C^*(\GG)$ is a RFD C*-algebra.
\end{proposition}

\begin{proof}
Since each $\Gamma_n$ is a finitely generated  nilpotent group, it is RF (\cite{Hirsch}), hence also MAP. By \cite[Theorem 4.3]{BekkaLouvet} $C^*(\Gamma_n)$ is RFD.
Proposition \ref{prop:indgrp} implies that $C^*(\GG)$ embeds into $\prod_{n \in \N} C^*(\Gamma_n)$, so it is also RFD.

Finally the isotropy subgroup $\GG_\infty = \Gamma$ is non-MAP.	
\end{proof}

The constructions  above can be upgraded (for example using the Cantor set as the base space $X$) so that they yield an amenable \'etale groupoid $\GG$ which is a bundle of isotropy groups such that
\begin{itemize}
\item[(i)] the set $\{x\in X: \GG_x \textup{ is not MAP}\}$ is dense in $X$;
\item[(ii)]	$C^*(\GG)$ is RFD.
	\end{itemize}

This of course does not yet imply that the second implication of Theorem \ref{theorem:main} cannot be reversed. Let us then state explicitly a question the answer to which would go towards settling the general problem.

\begin{quest} \label{quest}
Does there exist an amenable \'etale groupoid which is a bundle of isotropy groups, such that:
\item[(i)] no isotropy group is MAP;
\item[(ii)]	$C^*(\GG)$ is RFD?
\end{quest}

\section{A characterization of RFD crossed products} \label{Sec:crossedprod}

In this section we shift our attention to study the RFD property of $C^*$-algebraic crossed products.

Let $\alpha$ be an action of a discrete group $G$ on a $C^*$-algebra $A$, fixed throughout this section. Then there is a natural action  of $G$ on the family of all  representations of $A$  (denoted $\Rep(A)$),  defined by the formula
$$(g \cdot \rho)(a) := \rho(\alpha(g^{-1})(a)),  \;\;g \in G, \rho \in \Rep(A), a\in A.$$
It is clear that this action restricts to the set of irreducible representations of $A$, denoted $\Irr(A)$. It is easy to check that it also induces a $G$-action on $\Rep(A)/_\approx$, where $\approx$ denotes the unitary equivalence, on $\hat{A}$, i.e.\ the set of equivalence classes of irreducible representations in $\Rep(A)/_\approx$, and also on $\Prim(A)$, the set of primitive ideals of $A$ (note that the latter can be again viewed as a quotient space of $\Irr(A)$, but the equivalence relation is coarser: two irreducible representations yield the same element of $\Prim(A)$ if and only if their kernels coincide). As is customary, we equip $\Prim(A)$ with the Jacobson topology.

Given  a representation $\rho \in \Irr(A)$, and an action $\alpha: G \to \textup{Aut}(A)$ it may well happen that the class of $\rho$ in $\textup{Prim}(A)$ has a finite orbit, and the class of $\rho$ in $\hat{A}$ has an infinite orbit. This however does not happen for finite-dimensional representations, for a simple reason: if we have two irreducible non-degenerate representations $\rho,\pi$ in $\Irr (A)$ and we know that $\rho$ is finite-dimensional, then $\Ker(\rho) = \Ker(\pi)$ if and only if $\pi \approx \rho$. This is well-known and very easy to check, and will be used later without any comment. Note that  we will use the conventions of \cite{ExelLoring}, i.e.\ by a finite-dimensional representation we understand a possibly degenerate representation $\rho:A \to B(\Hil)$ whose essential space $\overline{\rho(A)\Hil}$ is finite-dimensional. Let us also record here the following fact (see for example \cite[Lemma 5.44, Remark 5.45]{DanaIan}).

\begin{lemma}\label{homeoPrim}
	If $A$ is a C*-algebra equipped with an action of a discrete group $G$, then $G$ acts on $\Prim(A)$ via homeomorphisms.
\end{lemma}

Finally recall that $G$ acts also on the state space $S(A)$, with $g \cdot \omega = \omega \circ \alpha_{g^{-1}}$ for $g \in G$, $\omega \in S(A)$. This action restricts to $P(A)$, the pure state space, but  also descends to $S(A)/_\approx$, where for two states $\omega_1, \omega_2 \in S(A)$ we write $\omega_1 \approx \omega_2$ if the GNS representations of $\omega_1$ and $\omega_2$ are unitarily equivalent. We say that a state $\omega \in S(A)$ is \emph{finite-dimensional} if its GNS representation acts on a finite-dimensional space.

Given $\rho \in \Rep(A)$ we denote by $G_{\rho}$ the stabilizer subgroup of the class of $\rho$ in $\Rep(A)/_\approx$, that is
\begin{equation}G_{\rho} = \{ g\in G \;|\; g\cdot\rho \approx\rho\}. \label{Grho}\end{equation}

\subsection{Approximations via periodic finite-dimensional representations}

Before we consider the general characterisation of the RFD property of crossed products, we need to establish connections between several properties expressing a possibility of approximating any representation/state of a given C*-algebra by relevant finite-dimensional and periodic objects.
 We begin with some basic remarks.

\begin{remark}\label{remark:trivial2}
	Suppose that $\rho:A \to B(\Kil)$ is a finite dimensional representation, which decomposes into irreducibles as $\rho\approx \rho_1 \oplus \cdots \oplus \rho_n$. If the orbit of $\rho$ in  $\Rep(A)/_\approx$ is finite, then for every $i=1,\ldots, n$ the orbit of $\rho_i$ in $\Rep(A)/_\approx$ is finite.
	Moreover a finite direct sum of finite dimensional representations with finite orbits in  $\Rep(A)/_\approx$ also has a finite orbit in $\Rep(A)/_\approx$.
	
\end{remark}
\begin{proof}
	Assume first that $G_\rho=G$. Then fix $g \in G$ and observe the following:
	\[\rho_1 \oplus \cdots \oplus \rho_n \approx \rho \approx g \cdot \rho \approx g \cdot\rho_1 \oplus \cdots \oplus g \cdot \rho_n. \]
	By the uniqueness of the decomposition into irreducibles (up to unitary equivalence and permutations), see \cite[Subsection 2.3.5]{Dixmier}, we have
	$g \cdot\rho_1 \approx \rho_k$ for some $k\in \{1, \ldots, n\}$. The claim follows.
	
	The argument in the general case of the first statement is completely analogous; and similarly one can show the second statement.
\end{proof}

\begin{remark}\label{remark:FE}
	Let  $F$ be the set of all finite-dimensional states on $A$ with finite orbits in $S(A)/_\approx$  and $E$ the set of all pure finite-dimensional states on $A$ with finite orbits in $P(A)/_\approx$.
	Then  $F$ is the convex hull of $E$.
\end{remark}

\begin{proof}  Let $\omega$ be a finite-dimensional state on $A$ with finite orbit in $S(A)/_\approx$ and let $(\pi,\Hil, \xi)$ be its GNS triple.  Since $\pi$ is a finite-dimensional representation, there are irreducible finite-dimensional representations $\rho_1, \ldots, \rho_k$ such that
	$\pi \approx \oplus_{i=1}^k \rho_i.$ By Remark \ref{remark:trivial2} each $\rho_i$ has finite orbit in  $\Rep(A)/_\approx$. Let $\xi = (\xi_1, \ldots, \xi_k)$ be the corresponding orthogonal decomposition of $\xi$ (into non-zero vectors).  For each $i \in \{1, \ldots,k\}$ set $\lambda_i = \|\xi_i\| $ and $\widetilde{\xi_i}: = \lambda_i^{-1} \xi_i$, so that $\|\widetilde{\xi_i}\|=1$. For any $a \in A$ we have
	\begin{equation}\label{eqClaim} \omega(a)  =  \langle\xi, \pi(a)\xi\rangle = \sum_{i=1}^k \langle\xi_i, \rho_i(a)\xi_i\rangle = \sum_{i=1}^k \lambda_i^{2} \langle\widetilde{\xi_i}, \rho_i(a) \widetilde{\xi_i}\rangle.\end{equation}
	Each of the states $\omega_i:=\omega_{\widetilde{\xi_i}} \circ \rho_i$  is pure, as   $\rho_i$ is irreducible; it is also easy to see that $\rho_i$ is a GNS representation of $\omega_i$. Thus each $\omega_i$ is finite-dimensional and has a finite orbit in  $S(A)/_\approx$. % because each $\rho_i$ is finite-dimensional and has finite orbit in  $\Rep(A)/_\approx$.
	Since $\xi$ is a unit vector,  by (\ref{eqClaim}) we decomposed $\omega$ into a convex combination of pure states with desired properties.
	
	On the other hand it is easy to see that a convex combination of states with finite orbits in $S(A)/_\approx$ itself has a finite orbit.
\end{proof}

Let $\Hil$ be a  Hilbert space, $\Hil_n\subseteq \Hil$ ($n\in \mathbb N$)  its closed subspaces, $\pi: A \to B(\Hil)$ and $ \pi_n: A \to B(H_n)$ representations. If for any $a\in A$
$$\pi_n(a) \to \pi(a) \;(SOT),$$
then we will say that {\it $\pi$ is approximated by $\pi_n$ in SOT} (that is the strong operator topology).

\begin{remark}\label{remark:afd}
	Suppose that for each $n \in \N$ we have a representation $\pi_n:A\to B(\Kil_n)$ which can be approximated in SOT topology by finite-dimensional representations. Then $\bigoplus_{n \in \N} \pi_n:A \to B(\bigoplus_{n \in \N} \Kil_n)$ can be approximated in the SOT topology by finite-dimensional representations. Further if $\pi:A \to B(\Hil)$ is a representation of a separable C*-algebra $A$ on a separable Hilbert space which is approximately unitarily equivalent to a representation $\rho:A \to B(\Kil)$ which can be approximated in SOT topology by a sequence of finite-dimensional representations, then $\pi$ itself can be approximated in SOT topology by a sequence of  finite-dimensional representations, each of which is equivalent to one of these approximating $\rho$.
\end{remark}
\begin{proof}
	The first part is standard, using finite direct sums of finite-dimensional representations approximating each $\pi_n$.
	
	For the second assume  that $\rho:A \to B(\Kil)$ can be approximated in the SOT topology by the net $(\rho_i:A \to B(\Kil))_{i \in I}$ of finite-dimensional representations and  that $(V_n)_{n \in \N}$ is a sequence of unitaries from $\Hil$ to $\Kil$ such that for every $a \in A$ we have $\pi(a) = \lim_{n \to \infty} V_n^* \rho(a) V_n$. Set $\pi_{n,i}:A \to B(\Hil)$, $\pi_{n,i} = V_n^* \rho_i(\cdot) V_n$.
	By the usual indexing argument using countable dense subsets in unit balls of $\Hil$ and $A$ it suffices to show that for any $\epsilon>0$, a finite set $F\subset A$  and $ G \subset \Hil$ we can find some $\pi_{n,i}$ such that $ \|\pi_{n,i}(a) \xi - \pi(a)\xi\| < \epsilon$ for all $a \in F$, $\xi \in G$. So we just consider the inequality
	\begin{align*} \|\pi_{n,i}(a) \xi - \pi(a)\xi\| &=  \| V_n^* \rho_i(a) V_n \xi - \pi(a)\xi \|\\&\leq
	\| V_n^* \rho_i(a) V_n \xi - V_n^* \rho(a) V_n \xi\| + \| V_n^* \rho(a) V_n \xi-  \pi(a)\xi \|,	
	\end{align*}
	and given $\epsilon>0$ first choose $n\in \N$ big enough so that the second part is smaller then $\epsilon/2$ for all $a \in F$, $\xi \in G$ and then for this $n$ choose $i\in I$ big enough so that the first part is also smaller than $\epsilon/2$ for $a \in F$, $\xi \in G$.
\end{proof}

%We are ready for the key definition of the class of actions we will consider.

%\begin{definition} 	An action of a discrete group on a C*-algebra $A$ will be called  \emph{approximately fd-periodic} if  every pure state of $A$ can be approximated in weak$^*$-topology by finite-dimensional  pure states whose orbits in $S(A)/_\approx$ are finite. \end{definition}

The next result is the key technical result of this subsection; we will see it further simplifies in the separable case. %A reader who is (justifiably!) frightened by the maze of implications below is advised to refer directly to Theorems \ref{thm:justactseparable} and \ref{th:maincrossed}.{\color{blue}!!!!!!!}

Given two representations $\pi, \rho$ of a C*-algebra we say that \emph{$\pi$ is a corner of $\rho$} if $\pi$ is a restriction of $\rho$ to an invariant subspace, or in other words, if $\pi$ is a direct summand of $\rho$.

\begin{theorem}\label{thm:justact}
	Let $G$ be a discrete  group, $A$ a C*-algebra and $\alpha$ an action of $G$ on $A$. The following conditions are equivalent:
	\begin{itemize}	
		\item[(i)] any representation of $A$ is a corner  of a representation which can be approximated in the SOT-topology by finite-dimensional representations whose orbits in $\Rep(A)/_\approx$ are finite; 	
		\item[(ii)]every element of $\Prim(A)$ can be approximated in the Jacobson topology by finite-dimensional elements of $\Prim(A)$ whose orbits are finite;
		\item[(iii)]  every state of $A$ can be approximated in weak$^*$-topology by finite-dimensional  states whose orbits in $S(A)/_\approx$ are finite;
		\item[(iv)] every pure state of $A$ can be approximated in weak$^*$-topology by finite-dimensional  pure states whose orbits in $P(A)/_\approx$ are finite.
\end{itemize}

\noindent The conditions above imply also that
\begin{itemize}

		\item[(v)] any infinite-dimensional irreducible representation of $A$ can be approximated in the SOT-topology by finite-dimensional irreducible  representations whose orbits in $\Rep(A)/_\approx$ are finite.
		
	\end{itemize}
	
\end{theorem}
\begin{proof}
	$(i)\Longrightarrow(ii)$:
	Let $\rho \in \Irr(A)$. By (i) the representation $\rho$ is a corner of a representation $\pi:A \to B(\Kil)$ which can be approximated in the SOT topology by finite-dimensional representations $\pi_i:A \to B(\Kil)$, $i \in I$, whose orbits in $\Rep(A)/_\approx$ are finite. Decompose each $\pi_i$ into irreducibles $\pi_i^1, \ldots, \pi_i^{n_i}$. By Remark \ref{remark:trivial2} each $[\pi_i^j]$ has finite orbit in $\Irr(A)/_\approx$, hence in $\Prim(A)$. Moreover for any $i_0 \in I$  we have
	\[ \bigcap_{i \in I, i \geq i_0} \bigcap_{j=1}^{n_i} \Ker (\pi_i^j) = \bigcap_{i \in I, i \geq i_0}  \Ker (\pi_i) \subset \Ker(\pi) \subset \Ker (\rho),\]
	which means that the net $(\Ker(\pi_i^j))_{i \in I, j=1, \ldots, n_i}$ approximates $\Ker(\rho)$ in the Jacobson topology.
	
	\medskip

	$(ii)\Longrightarrow (iii)$:
	Suppose that (iii) does not hold. Since by Remark \ref{remark:FE} finite-dimensional states with finite orbits form  a convex set, by Hahn-Banach theorem we can find an element $a \in A_+$, a state $\omega \in S(A)$ and number $M>0$ such that $\omega(a)>M$ and $\mu(a) \leqslant M$ for every finite-dimensional  state $\mu$ whose orbit in $S(A)/_\approx$ is finite.  Note that we must have $\|a\|> M$, so that there is a vector state $\tilde{\omega}$ associated to an irreducible representation of $A$ such that $\tilde{\omega}(a)>M$. By
	\cite[Theorem 1.4]{Fell} the assumption (ii) implies that there exist some states $\mu_i$ associated with finite-dimensional irreducible representations with finite orbits in $\Rep(A)/_\approx$ such that $\mu_i(a)>M$. But then $\mu_i$ are finite-dimensional states which  have finite orbits in $S(A)/_\approx$ and we have reached a contradiction.
	
		\medskip

$(iii)\Longrightarrow(iv) $:
Let $F$ and $E$ be as in Remark \ref{remark:FE}. Then (iii) implies that $\mathcal P(A) \subset \overline F$. Thus $$S(A) = \overline{conv \; \mathcal P(A)} \subset \overline{conv \; \overline F} = \overline F.$$  By Remark  \ref{remark:FE} $\overline F = \overline {conv \; E}$. Thus  $S(A) =  \overline {conv \; E}$.  Then by Milman's converse to Krein-Milman theorem
$\overline E$ contains all  extreme points of $S(A)$, that is all pure states. This is precisely what (iv) says.
	
\medskip

$(iv)\Longrightarrow(i) $: Let $\{a_{\alpha}\}_{\alpha\in I}$ be any dense subset of $A$. For any $\alpha\in I$ there is a pure state $\chi_{\alpha}\in P(A)$ such that $|\chi_{\alpha}(a_{\alpha})| \ge  \frac{\|a_{\alpha}\|}{2}$. Statement (iv) implies that there is a  finite-dimensional (pure) state $\tilde \chi_{\alpha} \in P(A)$  whose orbit in $S(A)/_\approx$ is finite and such that $|\tilde \chi_{\alpha}(a_{\alpha})|\ge \frac{\|a_{\alpha}\|}{4}$; denote its GNS representation by $\pi_{\alpha}$. Then  $\{\pi_{\alpha}\}_{\alpha\in I}$ is a separating family of finite-dimensional representations of $A$ with finite orbits in $\Rep(A)/_\approx$.

Now let $\rho$ be a representation of $A$ on a Hilbert space $\Hil$.  We are going to show that $\rho$ is a corner of representation that can be approximated by finite-dimensional representations in SOT topology.\footnote{We will use here arguments that are somewhat similar to arguments Hadwin used to obtain his lifiting characterization of RFD property \cite{Hadwin}.}  We can assume that $\rho$ is non-degenerate.  Let $\it m$ be any infinite cardinality such that $\it m \ge \dim \Hil$. For each $\alpha \in I$ let $\pi_{\alpha}^{(\it m)}$ be the direct sum of $\it m$ copies of $\pi_{\alpha}$ and let
$$\pi = \oplus_{\alpha\in I} \pi_{\alpha}^{(\it m)}.$$
Then the dimension of $\pi$ (that is the dimension of the Hilbert space on which it acts) is larger or equal than the dimension of $\rho$. Hence
$$\dim (\rho\oplus\pi)  = \dim \rho + \dim \pi = \max\{\dim \rho, \dim \pi\} = \dim \pi, $$
so we can consider $\rho\oplus\pi$ and $\pi$ as representations on the same Hilbert space (recall that if $k$ and $m$ are two cardinal numbers and at least one of them is infinite, then $k+m = \max\{k, m\}$).
Let $a\in A$. There is $\alpha\in I$ such that $\pi_{\alpha}(a)\neq 0$.  Then $\rank \pi(a) \ge \it m.$ On the other hand $\rank  \rho(a) \le \dim \Hil \le \it m.$ Hence
$$\rank (\rho \oplus \pi)(a) = \rank  \rho(a) + \rank  \pi(a) = \max\{\rank \rho(a), \rank \pi(a)\} = \rank  \pi(a).$$ By a version of Voiculescu's theorem for general, not necessarily separable, spaces and C*-algebras provided in \cite[Th. 3.14]{HadwinNonseparable} we conclude that $\rho\oplus \pi$ and $\pi$ are approximately unitarily equivalent. Since  by its construction  $\pi$ can be approximated in SOT topology by finite dimensional representations with finite orbits, by Remark \ref{remark:afd} (or rather its non-separable version) the same is true for  $\rho \oplus \pi$. Since $\rho$ is a corner of $\rho \oplus \pi$, the statement (i) is proved.

\medskip

To prove the last statement we show that $(iv)\Longrightarrow(v) $: Let $\rho: A \to B(\Kil)$ be an irreducible infinite-dimensional representation. Then $\rho = \pi_\omega$, for some pure state $\omega\in P(A)$. By (iv) we have finite-dimensional pure states  $(\omega_i)_{i \in I}$  with finite orbits in $S(A)/_{\approx}$
	such that  $\omega_i\longrightarrow_{i \in I} \omega$ in the weak*-topology.  Let $(\pi_{\omega_i}, \Hil_i, \Omega_i)$ and $(\pi_\omega, \Hil, \Omega)$ be the corresponding GNS triples.
	By \cite[proof of (a) $\Rightarrow$ (b) in Theorem 2.4]{ExelLoring} %(which is in fact one of key ingredients of the proof of that result)
	there exist coisometries $V_i^j : \Hil \to \Hil_{i}$
	and a net $(\pi_j)_{j \in \mathcal J}$ of representations of $A$ on $\Hil$ which converges to $\rho$ in the pointwise  SOT topology and such that each $\pi_j$ is of the form  $(V_i^j)^* \pi_{\omega_i}(\cdot)V_i^j$ for some $i \in I$.
 For every $i \in I$ since  $\omega_i$ is pure, finite-dimensional, and has finite orbit in $S(A)/_{\approx}$, the representation $\pi_{\omega_i}$ is irreducible, finite-dimensional and has finite orbit in $\Rep(A)/_{\approx}$ and the same is true for $(V_i^j)^* \pi_{\omega_i}(\cdot)V_i^j$.
	
	We need to add a comment here: the proof in \cite{ExelLoring} assumes in addition that $A$ is unital. This is however not an issue, as we can first extend $\pi_{\omega_i}, \pi_\omega$ to unital representations of the minimal unitisation $\tilde{A}$ and then note that the proof in \cite{ExelLoring} still yields for a given finite subset $X \subset \tilde{A}$ the existence of coisometries $V_i^X:\Hil \to \Hil_i$ such that for any $a \in X$ we have $\lim_{i \in I} (V_i^X)^*\tilde{\pi}_{\omega_i}(a) V_i^X\Omega = \tilde{\pi}_\omega (a) \Omega$. But then we can just consider what the last statement says about finite subsets of $A$ and continue the proof as in \cite{ExelLoring}.

\end{proof}

\begin{remark}
	\label{rem:triv}
	In the conditions (i) and (iii) of Theorem \ref{thm:justact} finite orbits can be replaced by trivial orbits. Indeed,  assume  that (i) holds.	As the relation `being a corner of' is transitive, it suffices to consider a representation $\rho$  of $A$ which can be approximated in the SOT-topology by $(\pi_i)_{i \in I}$, finite-dimensional representations whose orbits in $\Rep(A)/_\approx$ are finite.
	Define $\tilde{\rho}:= \bigoplus_{g \in G} g \cdot \rho$, so that we can view $\tilde{\rho}$ as a representation of $A$ on $\Hil_\rho \otimes \ell^2(G)$, of which $\rho$ is a corner. We will show that $\tilde{\rho}$ can be approximated in the SOT-topology by finite-dimensional representations whose orbits in $\Rep(A)/_\approx$ are trivial.
	
To this end we need to show that for all  finite sets $S\subset A$, $F \subset   \Hil_\rho \otimes \ell^2(G)$ and $\epsilon>0$ there exists a finite-dimensional representation $\pi:A \to B(\Hil_\rho \otimes \ell^2(G))$ whose orbit in $\Rep(A)/_\approx$ is trivial and such that
	\begin{equation}\|\tilde{\rho}(a) \xi - \pi(a)\xi \|<\epsilon, \;\;\; a \in S, \xi \in F. \label{approxlocal}\end{equation} 	
By a standard approximation (and modifying $\epsilon>0$) we can assume that all $\xi \in F$ %have norm smaller than $1$ and
are finitely supported with respect to $G$ (say in a non-empty finite set $Z$)%, and that for every $a\in S$ we have $\|a\|\leq 1$
.  Our initial  assumption implies that we can find $i \in I$ such that $\|(g \cdot\rho)(a) \xi_g -  (g \cdot\pi_i)(a)\xi_g \|<\epsilon'$ for all $g \in Z$, $ a\in S$, $\xi \in F$, where $\epsilon ' = \epsilon |Z|^{-\frac{1}{2}}$.

Choose a finite set $V \subset G$ of representatives of cosets $G/G_i$, where $G_i:= \{g \in G: g \cdot \pi_i \approx \pi_i\}$. Note that the representation $\tilde{\pi}_i:=\bigoplus_{g \in V} g \cdot \pi_i$ is a finite-dimensional representation of $A$ with trivial orbit in $\Rep(A)/_\approx$. Further choose another finite set $W \subset G$ with the following properties: the sets $(wV)_{w \in W}$ are pairwise disjoint and together cover $Z$. Finally define $\pi:= \bigoplus_{w \in W} \bigoplus_{g \in V} (wg)\cdot \pi_i\approx \bigoplus_{w \in W} w \cdot \tilde{\pi}_i$. The second expression shows that $\pi$ has a trivial orbit in  $\Rep(A)/_\approx$, the first allows us to view $\pi$ as a representation of $A$ on $\Hil_\rho \otimes \ell^2(G)$. We then have
\begin{align*} \|\tilde{\rho}(a) \xi - \pi(a)\xi \|^2 = \sum_{g \in Z} \|(g \cdot\rho)(a) \xi_g - (g \cdot\pi_i)(a)\xi_g \|^2 \leq \epsilon'^2 |Z| =\epsilon^2.
\end{align*}
Now assume that (iii) holds. Let $\omega \in S(A)$ be a state, and let $(\rho, \Omega, \Hil)$ be its GNS triplet. By what is proved above we have a representation $\pi:A \to B(\Kil)$ in which $\rho$ is a corner and a net of finite-dimensional representations $(\pi_i:A \to B(\Kil))_{i \in I}$ which approximate $\pi$ in the SOT topology and which have trivial orbits in $\Rep(A)_\approx$. Consider the non-negative vector functionals $\omega_i = \langle \Omega, \pi_i(\cdot) \Omega \rangle$. We have $\|\omega_i\| \leq 1$ and $\lim_{i \in I} \|\omega_i\|=1$. Assume then without loss of generality that $\omega_i \neq 0$ and consider the normalised version of $\omega_i$, denoted further $\tilde{\omega}_i$; we then have $\lim_{i \in I}\tilde{\omega}_i = \omega$. Fix then $i \in I$. Naturally $\tilde{\omega}_i$ is still a vector functional (for a rescaled vector $\Omega_i$). Set $\Kil_i:= \overline{\pi(A) \Omega_i}$, let $P_i:\Kil \to \Kil_i$ be the orthogonal projection onto $\Kil_i$ and let $\tilde{\pi}_i:A \to B(\Kil_i)$ be the corestriction of $\pi_i$.  It is easy to check that $(\tilde{\pi}_i, P \Omega_i, \Kil_i)$ is a GNS triple for $\tilde{\omega}_i$; thus the latter is a finite-dimensional state. Moreover it is easy to see that for every $g \in G$ if we have $g \cdot \pi_i \approx \pi_i$, then naturally also $g \cdot \tilde{\pi}_i \approx \tilde{\pi}_i$, and further $g \cdot \tilde{\omega}_i \approx \tilde{\omega}_i$ (as $(g \cdot\tilde{\pi}_i, P \Omega_i, \Kil_i)$  is a GNS triple for $ g \cdot\tilde{\omega}_i$).
\end{remark}

%In the separable case, where we can make use of the Voiculescu Theorem, several conditions above turn out to be equivalent, as we will show in the next theorem. We suspect that this might be the case in general. For brevity we will write that a representation is \emph{separable} if it acts on a separable Hilbert space.

The condition (i) in Theorem \ref{thm:justact} can be improved when $A$ is separable, as we show in the next theorem.

\begin{theorem}\label{thm:justactseparable}
	Let $G$ be a discrete  group, $A$ a separable C*-algebra and $\alpha$ an action of $G$ on $A$.  Then the equivalent conditions from Theorem \ref{thm:justact}  are equivalent to the following:

\medskip

 any  representation of $A$  on a separable infinite-dimensional Hilbert space can be approximated in the SOT-topology by finite-dimensional representations whose orbits in $\Rep(A)/_\approx$ are trivial.

\end{theorem}
\begin{proof}
		Let $\rho:A \to B(\Hil)$ be a representation of $A$ on an infinite-dimensional  separable Hilbert space. Condition (i)  implies that $\rho$ is a corner inside a representation $\pi:A \to B(\Kil)$ which can be approximated in the SOT topology by finite-dimensional representations $\pi_i$ with trivial orbits in $\Rep(A)/_\approx$.
	As $\rho$ is infinite-dimensional, by a result of Hadwin (see \cite[Lemma 2.4]{CS}) there exist unitary operators $W_k:\Hil \to \Kil$, $k \in \N$, such that
	$$\rho = SOT-lim \; W_k^*\pi (\cdot)W_k.$$
	Let then $\rho_{k,i}:=W_k^* \pi_i(\cdot) W_k$, $k \in \N, i \in I$. Since all $\pi_i$ have trivial orbits in $\Rep(A)/_\approx$, so do all $\rho_{k,i}$.
	Arguing as in the proof of the second part of the Remark \ref{remark:afd} we can extract a subsequence from the set $\{\rho_{k,i}: k \in \N, i \in I\}$ which converges to $\rho$ in the SOT topology.	
	
For the converse implication start from $\rho$, a representation of $A$. Since any representation is the direct sum of cyclic ones, we can assume that $\rho$ is cyclic. Then it acts on a separable Hilbert space. If the carrier space is infinite-dimensional, then the statement follows. If it is finite-dimensional, then $\rho^{(\infty)}$, that is  the direct sum of countably many copies of $\rho$ (in which $\rho$ is a corner), is an infinite-dimensional representation on a separable Hilbert space and hence can be approximated by finite-dimensional representations whose orbits in $\Rep(A)/_\approx$ are trivial.
\end{proof}

\begin{remark}\label{conditionii} Since in the rest of the paper we will very often use the condition (ii) of Theorem \ref{thm:justact}, let us discuss its meaning. If $\mathcal F$ is a family of irreducible representations of $A$, then, by definition of Jacobson topology,   $\{\ker \rho\;|\; \rho\in \mathcal F\}$ is dense in $\Prim(A)$ if and only $\mathcal F$ separates points of $A$.
Thus the condition (ii) of Theorem \ref{thm:justact} means that finite-dimensional (irreducible) representations whose orbits in $\Prim(A)$ are finite separate points of $A$. In the case
of trivial $G$ it just means to be RFD.
\end{remark}

\begin{remark}\label{DynamicalVersionEL} Recall that (the main part of) Exel-Loring characterization of RFD C*-algebras states that  a  C*-algebra is  RFD iff any its representation can be approximated in SOT-topology by finite-dimensional ones \cite{ExelLoring}. The condition (i) of Theorem \ref{thm:justact} also means approximation of a given representation in SOT-topology by finite-dimensional representations which in addtion have finite orbits, only those representations live on subspaces of a  bigger space than the space of  given representation. Thus the implication  (ii)$\Longrightarrow$(i) of Theorem \ref{thm:justact} can be considered as a dynamical version (although not exactly a generalization) of Exel-Loring result. Moreover in the separable case it becomes an  actual generalization by Theorem \ref{thm:justactseparable}. Indeed  it says that if $G$ acts on $A$  with sufficiently many finite-dimensional irreducible representations with finite orbits in $\Prim(A)$ (which in the case of trivial $G$ just means that $A$ is RFD), then we can approximate in SOT-topology  by finite-dimensional representations which in addition have finite orbits.
\end{remark}

\subsection{Finite-dimensional (and not only) representations of $A \rtimes G$}

In this subsection we will construct some representations of the universal crossed product algebra $A \rtimes G$. Recall that such representations arise from covariant pairs, i.e.\ from pairs $(\pi,U)$, where $\pi:A\to B(\Kil)$ and $U:G \to B(\Kil)$ are respectively a representation of $A$ and a unitary representation of $G$ on the same Hilbert space $\Kil$ such that for every $g \in G$ and $a \in A$ we have $(g\cdot \pi)(a) = U_g \pi(a) U_g^*$; see \cite{Dana1} for the details.

We begin with a simple lemma; recall the definition of $G_\rho$ from \eqref{Grho}.

\begin{lemma}\label{lem:implement}
	Let $\rho:A \to B(\Hil_\rho)$ be an irreducible representation. Then there exists a unitary representation $V:G_\rho \to B(\Hil_\rho \otimes \overline{H_\rho})$ such that
	\begin{equation} g \cdot \rho(\cdot) \otimes I_{\overline{\Hil_\rho}} = V(g)^* (\rho(\cdot) \otimes I_{\overline{\Hil_\rho}}) V(g), \;\;\; g \in G_\rho. \label{implementrep}\end{equation}
\end{lemma}

\begin{proof}
	%	Note first that we can assume that $\rho$ is itself irreducible (once we establish the result in this case we can simply use the direct product decomposition).
	Definition of $G_\rho$ implies that for any $g \in G_\rho$ we have a unitary $\widetilde{V_g} \in B(\Hil)$ such that
	\[ g \cdot \rho = \widetilde{V_g}^*\rho(\cdot) \widetilde{V_g}.\]
	Considering $g_1, g_2 \in G_\rho$ we obtain the following  	
	\[ \widetilde{V_{g_1g_2}}^*\rho(\cdot) \widetilde{V_{g_1g_2}} =   \widetilde{V_{g_2}}^* \widetilde{V_{g_1}}^* \rho(\cdot) \widetilde{V_{g_1}} \widetilde{V_{g_2}}. \]
	This means that  $\widetilde{V_{g_2}}^* \widetilde{V_{g_1}}^* \widetilde{V_{g_1g_2}} \in \rho(A)'$. As $\rho$ is  assumed to be irreducible, we deduce that
	\[ \widetilde{V_{g_1g_2}} = \lambda(g_1, g_2) \widetilde{V_{g_1}} \widetilde{V_{g_2}},\]	
	where $\lambda(g_1, g_2) \in \mathbb{T}$.
	Set then $V(g) = \widetilde{V_g} \otimes \overline{\widetilde{V_g}} \in B(\Hil_\rho \otimes\overline{ \Hil_\rho})$, where naturally for any $\zeta \in \Hil_\rho$ we have $\overline{\widetilde{V_g}} \overline{\zeta} = \overline{\widetilde{V_g} \zeta}$. The formula \eqref{implementrep} is easy to check; moreover for any $g_1, g_2 \in G_\rho$ we have
	\begin{align*} V(g_1 g_2) &= \widetilde{V_{g_1 g_2}} \otimes \overline{\widetilde{V_{g_1 g_2}}} =
		\lambda(g_1, g_2) \widetilde{V_{g_1}} \widetilde{V_{g_2}} \otimes \overline{\lambda(g_1, g_2) \widetilde{V_{g_1}} \widetilde{V_{g_2}}} =  \widetilde{V_{g_1}} \widetilde{V_{g_2}} \otimes \overline{ \widetilde{V_{g_1}} \widetilde{V_{g_2}}} \\&= V(g_1)V(g_2),	
	\end{align*}	
	so that $V:G \to B(\Hil_\rho \otimes\overline{ \Hil_\rho})$ is a unitary representation.
\end{proof}

The next lemma introduces an additional representation $U$ of $G_\rho$ and uses the above fact to produce a representation of $A \rtimes G$ out of $\rho$ and $U$.

\begin{proposition} \label{prop:repconstruction}
	Let $\rho:A \to B(\Hil_\rho)$ be an irreducible representation and let $U:G_\rho \to B(\Kil)$ for some Hilbert space $\Kil$	be a unitary representation. Choose representatives  of cosets in $G/G_\rho$, and denote them $(s_i)_{i \in I}$. Write $K_\rho= \Hil_\rho \otimes \overline{\Hil_\rho}$, and let $V:G_\rho \to B(\Kil_\rho)$ be a representation constructed in Lemma \ref{lem:implement}. Set $\Hil = \ell^2(I) \otimes \Kil_\rho \otimes \Kil$. Finally for $a \in A$, $\xi \in \Kil_\rho$, $\eta \in \Kil$, $g \in G$, $i \in I$ set
	\[ \pi(a) (e_i \otimes \xi \otimes \eta) = e_i \otimes ((s_i \cdot \rho)(a) \otimes I_{\overline{\Hil_\rho}}) \xi \otimes \eta,\]
	\[ L(g) (e_i \otimes \xi\otimes \eta) = e_j \otimes V(t) \xi  \otimes U(t) \eta,\]
	where $j\in I$ and $t \in G_\rho$ are such that $gs_i = s_jt$ (note that this determines $j$ and $t$ uniquely).
	Then the above formulas determine a covariant pair $(\pi, L)$ of $(A, G, \alpha)$ on $\Hil$. We will write $\gamma_{\rho,U}:A \rtimes G \to B(\Hil)$ for the representation associated to the pair $(\pi, L)$.
\end{proposition}

\begin{proof}
	For simplicity whenever $a \in A$  we will write $\tilde{\rho}(a) = \rho(a) \otimes I_{\overline{\Hil_\rho}} \in
	B(\Hil_\rho \otimes\overline{ \Hil_\rho})$.
	Lemma \ref{lem:implement} says that for every $t \in G_\rho$
	\[ t \cdot \tilde{\rho} = V(t)^* \tilde{\rho}(\cdot) V(t).\]
	The first observation we need is the following: for any $j \in I$, $t \in G_\rho$ we still have
	\[ (s_j t \cdot) \tilde{\rho} = V(t)^* (s_j \cdot \tilde{\rho}) V(t), \]
	or in other words
	\begin{equation} V(t) (s_j t \cdot \tilde{\rho})(\cdot)  =  (s_j \cdot \tilde{\rho})(\cdot) V(t); \label{shiftedimplement}\end{equation}
	this is indeed easy to check.
	We then verify the following:
	\begin{itemize}
		\item $\pi$ is a representation: this is clear, as $\pi$ can be viewed simply as a direct sum of the form $\bigoplus_{i \in I} (s_i \cdot \tilde{\rho}) \otimes \id_{\Kil}$;
		\item $L:G \to B(\Hil)$ defines a representation: again, it is easy to see that the linear extension of the formula above defines a unitary $L(g)$, and also if $gs_j = s_k r$ and $hs_i= s_jt$ ($g,h\in G$, $r,t \in G_\rho$, $i,j,k \in I$) we have $(gh) s_i = s_k(rt)$, so that
		\begin{align*} L(gh) (e_i \otimes \xi \otimes \eta) &= e_k \otimes V(rt) \xi \otimes U(rt)(\eta),\end{align*}
		whereas
		\begin{align*} L(g)(L(h) (e_i \otimes \xi \otimes \eta) )&= L(g) (e_j \otimes V(t) \xi \otimes U(t)\eta) =  e_k \otimes V(r)V(t) \xi \otimes U(r)U(t)(\eta) ;\end{align*}
		\item the pair $(\pi, L)$ forms a covariant representation. To that end we just compute ($a\in A, g \in G, i \in I, \xi \in \Kil_\rho, \eta \in \Kil$, with $j \in I$ and $t \in G_\rho$ such that $gs_i = s_j t$):
		\begin{align*} L(g) \pi(a) (e_i \otimes \xi \otimes \eta) &= L(g) (e_i \otimes (s_i \cdot \tilde{\rho})(a) \xi \otimes \eta) =   e_j \otimes V(t) (s_i \cdot \tilde{\rho})(a) \xi \otimes U(t) \eta,\end{align*}		
		whereas		
		\begin{align*} \pi(\alpha_g(a)) L(g) & (e_i \otimes \xi \otimes \eta) = \pi(\alpha_g(a)) (e_j \otimes V(t) \xi  \otimes U(t) \eta) \\&=   e_j \otimes ((s_j \cdot \tilde{\rho})(\alpha_g(a))) V(t)  \xi \otimes U(t) \eta =  e_j \otimes V(t)((s_j t \cdot \tilde{\rho})(\alpha_g(a)))   \xi \otimes U(t) \eta \\& =   e_j \otimes V(t)((gs_i \cdot \tilde{\rho})(\alpha_g(a)))   \xi \otimes U(t) \eta =
			e_j \otimes V(t)((s_i \cdot \tilde{\rho})(a))   \xi \otimes U(t) \eta,\end{align*}			
	\end{itemize}
	where we have used \eqref{shiftedimplement} in the third equality.
\end{proof}

We will now compute the relevant trace of the images of the elements of $C_c(G,A)$ with respect to $\gamma_{\rho,U}$.

\begin{lemma}\label{lem:trace}
	Let $\rho:A \to B(\Hil_\rho)$ be an irreducible representation and let $U:G_\rho \to B(\Kil)$ for some Hilbert space $\Kil$	be a unitary representation. The representation  $\gamma_{\rho,U}:G \rtimes A \to B(\Hil)$ constructed in Proposition \ref{prop:repconstruction} is finite-dimensional if and only if $\Hil_\rho$ and $\Kil$ are finite-dimensional and the orbit of $[\rho]$ in $\Rep(A)/_{\approx}$ is finite. In that case for any $a \in A$ and $g \in G$ we have
	\[ \textup{tr} (\gamma_{\rho,U} (a\delta_g)) = \frac{1}{|I|} \sum_{i \in I: s_i^{-1}gs_i \in G_\rho}
	\textup{tr} (( (s_{i} \cdot \rho)(a) \otimes I_{\overline{\Hil_\rho}})V(s_i^{-1}g s_i))	\textup{tr}(U(s_i^{-1}g s_i)).\]	
\end{lemma}
\begin{proof}
	The first part of the lemma is obvious (as the cardinality of the orbit in question equals the cardinality of the set $I$ in the last proposition).
	
	Use the notation of the last proposition and fix orthonormal bases $\{f_1,\cdots, f_{N^2}\}$ in $\Kil_\rho$ and $\{h_1,\cdots, h_{M}\}$, with $N=\textup{dim } \Hil_\rho$, $M = \textup{dim } \Kil$. Then (for every $i \in I$ setting $j_i\in I$ and $t_i \in G_\rho$ so that $gs_i = s_{j_i}t_i$)
	\begin{align*} \textup{tr} (\gamma_{\rho,U} &(a\delta_g)) =
		\frac{1}{|I|N^2 M} \sum_{i \in I} \sum_{k=1}^{N^2} \sum_{l=1}^M
		\langle e_i \otimes f_k \otimes h_l, e_{j_i} \otimes( (s_{j_i} \cdot \rho)(a) \otimes I_{\overline{\Hil_\rho}})V(t_i)f_k \otimes U(t_i)h_l \rangle
		\\&= 	\frac{1}{|I|N^2 M} \sum_{i \in I: gs_i \in s_i G_\rho}
		\sum_{k=1}^{N^2} \sum_{l=1}^M 	
		\langle f_k \otimes h_l,  ( (s_{i} \cdot \rho)(a) \otimes I_{\overline{\Hil_\rho}})V(s_i^{-1}g s_i)f_k \otimes U(s_i^{-1}g s_i)h_l \rangle
		\\&= \frac{1}{|I|} \sum_{i \in I: gs_i \in s_i G_\rho}
		\textup{tr} (( (s_{i} \cdot \rho)(a) \otimes I_{\overline{\Hil_\rho}})V(s_i^{-1}g s_i))	\textup{tr}(U(s_i^{-1}g s_i)).
	\end{align*}	
	
\end{proof}

We are now ready for the key estimate. The following lemma is the analogue of Lemma \ref{mainrevisited} in Section 2. Let $\Phi:G \rtimes A\to A$ denote the standard conditional expectation and $C_c(A, G)$ the convolution algebra. As usual we denote by $\tr$ the normalised trace on any matrix algebra.

\begin{lemma}\label{MainLemma} Let a discrete MAP group $G$ act on a C*-algebra  $A$,  and let $\rho$ be a finite--dimensional irreducible representation of $A$ such that the orbit of $[\rho]$ in $\Rep(A)/_{\approx}$ is finite, say given by  $\{s_1\cdot [\rho], s_2\cdot[\rho], \ldots, s_N\cdot[\rho]\},$ with $s_1=e, s_2, \ldots, s_N \in G$. Then for any  $b\in C_c(G, A)$ and $\epsilon >0$ there is a finite-dimensional representation $\tilde\pi$ of $A\rtimes G$ such that
	$$\left|\frac{\tr ((s_1 \cdot\rho)(\Phi(b)))+\ldots + \tr((s_N \cdot\rho)(\Phi(b)))}{N} - \tr\,  \tilde\pi(b)\right| < \epsilon.$$
\end{lemma}
\begin{proof}
	Write $b$ as $$b = \sum_{g\in E} a_g g,$$ where $E\subset G$ is finite, $e \in E$, $ a_g \in A$ for each $g \in E$. Let
	\begin{equation}\label{M} M = \max_{g\in E} \|a_g\|.\end{equation}
	Since $G$ is MAP, there is  a finite-dimensional representation   $u$
	of $G$ such that $u(h) \neq \id$, for any $e\neq h\in G_{\rho}\bigcap \{s_l^{-1}ts_l\;|\; l\le N,\; t\in E\}$. Let $\chi$ be the trivial representation of $G$. Then we have
	$$|\tr ( u\oplus \chi)(h)|\neq 1,$$ for any $e\neq h\in G_{\rho}\bigcap \{s_l^{-1}ts_l\;|\; l\le N,\; t\in E\}$. Taking an appropriate tensor power of $ u \oplus \chi$ we obtain a finite-dimensional representation $U$ of $G$ such that
	\begin{equation}|\tr (U(h))|<\frac{\epsilon}{NM|E|},\label{trestimate}\end{equation} for any $e\neq h\in G_{\rho}\bigcap \{s_l^{-1}ts_l\;|\; l\le N,\; t\in E\}$.
	Set $\tilde{\pi}:=\gamma_{\rho, U}$, following Proposition \ref{prop:repconstruction}.
	
	By Lemma \ref{lem:trace} we then have the following:
	
	\begin{align*} \tr (\tilde \pi(b)) &= \sum_{g\in E}\tr (\gamma_{\rho, U}(a_gL(g)) \\&= \frac{1}{N} \sum_{g\in E} \sum_{i \in I: s_i^{-1}gs_i \in G_\rho}
		\textup{tr} (( (s_{i} \cdot \rho)(a_g) \otimes I_{\overline{\Hil_\rho}})V(s_i^{-1}g s_i))	\textup{tr}(U(s_i^{-1}g s_i))
		\\&= \frac{1}{N} \sum_{i=1}^N  \textup{tr} (( (s_{i} \cdot \rho)(a_e) \otimes I_{\overline{\Hil_\rho}}) \\&\;\;+   \frac{1}{N} \sum_{g\in E, g \neq e} \sum_{i \in I: s_i^{-1}gs_i \in G_\rho}
		\textup{tr} (( (s_{i} \cdot \rho)(a_g) \otimes I_{\overline{\Hil_\rho}})V(s_i^{-1}g s_i))	\textup{tr}(U(s_i^{-1}g s_i)).
	\end{align*}
	
	By \eqref{M} and \eqref{trestimate} we can  estimate the absolute value of the second sum by $\epsilon$. Thus the fact that $a_e = \Phi(b)$ and our normalization convention finally yields
	
	$$\left|\frac{\tr ((s_1 \cdot\rho)(\Phi(b)))+\ldots + \tr((s_N \cdot\rho)(\Phi(b)))}{N} - \tr\,  \tilde\pi(b)\right| < \epsilon.$$
	
\end{proof}

\begin{remark}\label{remark:trivial}
	Suppose that $\pi:A\rtimes G \to B(\Kil)$ is a representation, and $\rho:=\pi|_A$. Then $G_\rho=\{g: g \cdot \rho \approx \rho\}=G$.
\end{remark}
\begin{proof}
	As we are considering the restriction of a representation of the crossed product the unitary equivalence between $\rho$ and $g \cdot \rho$ is simply implemented by the relevant $L(g)$.
\end{proof}

\subsection{Characterisation of the RFD property of crossed products}

We are finally ready for the main result of this section. We will use the notion of an \emph{amenable action} due to Claire Anantharaman-Delaroche (\cite[D\'efinition 4.1]{Claire}, see also \cite{TakaYuhei}). The theorem will be formulated for the universal crossed products, as we need to be able to construct representations from covariant pairs, as in Proposition \ref{prop:repconstruction}, but the assumption of the amenability of the action implies that the reduced and universal crossed products coincide, which indirectly plays a role in the proof -- see Remark \ref{rem:nonamen} below.

The characterization of the RFD property of crossed products below is given in terms of primitive ideals but we recall that by Theorem \ref{thm:justact} it admits reformulations in terms of states, pure states, representations etc. These reformulations are sometimes very useful (see e.g. subsection 5.3 on free product actions).

\begin{theorem} \label{th:maincrossed}
	Let $G$ be a discrete  group, $A$ a C*-algebra and $\alpha$ an amenable action of $G$ on $A$. Then the following conditions are equivalent:
	\begin{itemize}
		\item [(i)] $A\rtimes G$ is RFD;
		%\item[(ii)] $G$ is MAP and any irreducible representation of $A$ is a corner of a representation which can be approximated in the SOT-topology by finite-dimensional representations whose orbits in $\Rep(A)/_\approx$ are trivial;
		\item[(ii)] $G$ is MAP and every element of $\Prim(A)$ can be approximated in the Jacobson topology by finite-dimensional elements of $\Prim(A)$ whose orbits are finite.	
		%\item[(iv)] $G$ is MAP and every pure state of $A$ can be approximated in weak$^*$-topology by finite-dimensional  states whose orbits in $S(A)/_\approx$ are trivial;
		%\item[(v)] $G$ is MAP and the action $\alpha$  is approximately fd-periodic.
		% \item[(v)] $G$ is MAP and any infinite-dimensional irreducible representation of $A$ %on a separable Hilbert space
		%can be approximated in the SOT-topology by finite-dimensional irreducible  representations whose orbits in $\Rep(A)/_\approx$ are finite.
	\end{itemize}
\end{theorem}
\begin{proof} $(i)\Rightarrow (ii)$:
		As we have $\C[G] \subset M(A\rtimes G)$, by Remark \ref{rem:mult} it is clear that $G$ must be MAP. So, by Remark \ref{conditionii} it is sufficient to show that finite-dimensional irreducible representations of $A$ with finite orbits in $\Prim(A)$ separate points of $A$. Since $A\rtimes G$ is RFD, the restrictions of finite-dimensional representations of $A\rtimes G$ to $A$ separate points of $A$. The restriction of a representation of $A\rtimes G$ to $A$ has trivial orbit in $\Rep(A)/_\approx$. Hence for any $a\neq 0$ there is a finite-dimensional representation $\rho$ of $A$ with finite orbit in $\Rep(A)/_\approx$ such that $\rho(a)\neq 0$.  Passing to an appropriate irreducible summand of $\rho$ and using Remark \ref{remark:trivial2} we complete the proof.

	\medskip

	$(ii)\Rightarrow (i)$:
	Let $0\neq x$ be a positive element in $A\rtimes G$. Since the action is amenable, the conditional expectation $\Phi: A\rtimes G \to A$ is faithful (it is faithful on the reduced crossed product by \cite[Proposition 4.1.9]{BO} and amenability of the action guarantees that the full and reduced crossed products coincide, see (\cite[Proposition 4.8]{Claire})). Hence $\Phi(x)\neq 0$. 	

Then (ii) implies that there exists an irreducible  finite-dimensional representation $\rho: A \to B(\Hil_\rho)$ with finite orbit in $\Prim(A)$ such that   $\rho(\Phi(x))\neq 0$.
	 Say that this orbit is given by $\{s_1\cdot [\rho], s_2\cdot[\rho], \ldots, s_N\cdot[\rho]\},$ with $s_1=e, s_2, \ldots, s_N \in G$.  We then naturally have (as $s_1=e$)
	\[ \frac{(s_1 \cdot\rho)(\Phi(x))+\ldots + (s_N \cdot\rho)(\Phi(x))}{N}\gneq 0\]
	so also
	\[ C: = \frac{\tr ((s_1 \cdot\rho)(\Phi(x)))+\ldots + \tr((s_N \cdot\rho)(\Phi(x)))}{N}>0.\]

	There is $b\in C_c(G, A)$ such that $$\|x-b\|\le \frac{C}{4}.$$ For $\epsilon = \frac{C}{4}$, $b$ and $ \rho$ let  $\tilde\pi: A \rtimes G \to B(\Kil)$ be a finite-dimensional representation satisfying the requirements of Lemma \ref{MainLemma}. Then
	\begin{align*} \tr (\tilde\pi(x))& \ge \tr (\tilde\pi(b)) - |\tr (\tilde\pi(b-x))|
		\\&  \ge \tr \left(\frac{ \rho(\Phi(b))+\ldots + (s_N \cdot\rho)(\Phi(b))}{N} \right)  - \epsilon - |\tr (\tilde\pi(b-x))| \\ &=\tr \left(\frac{ \rho(\Phi(x))+\ldots + (g_N\cdot\rho)(\Phi(x))}{N} \right)  + \tr \left(\frac{ \rho(\Phi(b-x))+\ldots + (g_N\cdot\rho)(\Phi(b-x))}{N} \right) \\ &- \epsilon - |\tr (\tilde\pi(b-x))|\;  \ge \; C - \frac{3C}{4} > 0.
	\end{align*}
	Therefore $\tilde \pi(x) \neq 0$. Thus finite-dimensional representations of $A\rtimes G$ separate positive elements and hence all elements of $A\rtimes G$.
	
\end{proof}

\begin{remark}\label{rem:nonamen}
	The proof above could be also formulated for non-necessarily amenable actions. Condition (i) would then have to be replaced by saying that the reduced norm on the convolution algebra $C_c(G;A)$ is dominated by the supremum over the norms given by all finite-dimensional representations (see \cite[Proposition 1]{BekkaForum} for an analogous statement in the context of trivial $A$).
\end{remark}

\subsection{Some consequences}

Theorem \ref{th:maincrossed} formally generalises \cite[Theorem 4.3]{BekkaLouvet} and its predecessor \cite[Proposition 1]{BekkaForum}; and indeed, as already mentioned in the introduction, the ideas present in the proofs of these results were an inspiration for a proof of Theorem \ref{th:maincrossed}. Theorem \ref{th:maincrossed} also contains \cite[Cor. 4.1.11]{TomiyamaBook} as a particular case and also formally yields a different proof of Corollary  \ref{TransformationGroupoid}, although the methods we used in the proofs of Theorems \ref{theorem:main}  and  \ref{th:maincrossed} are ultimately similar.

Also, in addition to dynamical versions of Exel-Loring result discussed in Remark \ref{DynamicalVersionEL},  Theorem \ref{th:maincrossed}  gives one more related statement.

%The implication  (ii)$\Longrightarrow$(v) of Theorem \ref{thm:justact} can be considered as a dynamical version of their result as it says that if $G$ acts on $A$ amenably with sufficiently many finite obits in $\Prim(A)$, then we can approximate an irreducible representation by finite-dimensional representations which in addition have finite orbits.
%The following corollary can be considered as a dynamical version %(although not exactly a generalization)  of their result.

\begin{corollary}
	Let $G$ be a discrete  group, $A$ a C*-algebra and $\alpha$ an  action of $G$ on $A$. If 	 $A\rtimes G$ is RFD, then
	any infinite-dimensional irreducible representation of $A$ can be approximated in the SOT-topology by finite-dimensional irreducible  representations whose orbits in $\Rep(A)/_\approx$ are finite.
\end{corollary}
\begin{proof}
	The statement is a combination of the implication (i)$\Longrightarrow$(ii) of Theorem \ref{th:maincrossed} (note that this implication does not require amenability of action) and of the implication  (ii)$\Longrightarrow$(v) of Theorem \ref{thm:justact}.
%	
%	The second statement is now an easy consequence of the fact that a norm of an element $a \in A$ equals the $\sup\{\pi(a)\}$, where the supremum is taken over irreducible representations.
\end{proof}

In the case of trivial $G$ we obtain yet a new characterisation of the RFD property of C*-algebras (compare \cite[Theorem 2.4]{ExelLoring}).

\begin{corollary}\label{RFDnew}
A $C^*$-algebra $A$ is RFD if and only if any irreducible representation of $A$ can be approximated in the SOT-topology by finite-dimensional irreducible  representations.
\end{corollary}
\begin{proof}
The forward implication is a direct consequence of the previous corollary (since the case of a finite-dimensional  irreducible representation follows trivially here); the converse follows from the fact that a norm of an element $a \in A$ equals the $\sup\{\pi(a)\}$, where the supremum is taken over irreducible representations.
\end{proof}

We finish this subsection by a few more corollaries related to the case of FDI C*-algebras, i.e.\ C*-algebras whose all irreducible representations are finite-dimensional (see \cite[Theorem 4.4]{CSFDI} for equivalent characterisations). Recall Lemma \ref{homeoPrim}, and note that  if $\hat{A}$ is Hausdorff, $\Prim(A)= \hat{A}$ is a locally compact space (see \cite[Subsection 5.1]{DanaIan}). Thus we obtain the following `commutative' characterisation.

\begin{corollary} \label{cor:comm}
	Suppose that $G$ is an amenable discrete MAP group, $A$ is an FDI C*-algebra with Hausdorff spectrum (so for example a homogeneous C*-algebra or a unital continuous trace C*-algebra) and we have an action $\alpha:G \to \textup{Aut}(A)$. Then the following are equivalent:
	\begin{itemize}
		\item[(i)] $A \rtimes G$ is RFD;
		\item[(ii)] the induced action $\hat{\alpha}$ of $G$ on $\hat{A}$ has a dense set of periodic points;
		\item[(iii)] $C_0(\hat{A})\rtimes_{\hat{\alpha}} G$ is RFD.
	\end{itemize}	
\end{corollary}
\begin{proof}
		First note that the spectrum of a continuous trace C*-algebra (a particular case of which is a homogeneous C*-algebra) is Hausdorff  \cite[section IV.1]{Blackadar}.
		
		In view of the comments before the statement the corollary itself is an immediate consequence of the equivalence (i)$\Longleftrightarrow$(ii) in Theorem \ref{th:maincrossed}, applied once to the action of $G$ on $A$, and then to the action $\hat{\alpha}$ of $G$ on $\hat{A}$.		
\end{proof}

One could also ask about a natural counterpart of the above result in the non-Hausdorff case. We will present below a significantly weaker statement valid without the Hausdorff assumption.

Recall that any topological space $X$ admits a unique (up to homeomorphism) Stone-\u{C}ech compactification (so in particular a compact Hausdorff space) $\beta X$. The universal property defining $\beta X$ guarantees that any action of a discrete group $G$ on $X$ by homeomorphisms uniquely determines an action of $G$ on $\beta X$ by homeomorphisms; if $\beta: X \to \beta X$ denotes the canonical continuous map, then for any $g \in G$ and $x \in X$ the induced action is determined by the natural formula $g (\beta x) = \beta (gx)$.

Results of \cite{Nilsen} show that in fact for any C*-algebra $A$ we have $C(\beta \Prim(A)) \cong ZM(A)$. Thus we have the following corollary (we can check that the identifications of \cite{Nilsen} guarantee that if we begin with an  action of $G$ on $A$ by automorphisms then the two natural paths of constructing an action on $ZM(A)$ yield the same answer).

\begin{corollary} \label{cor:commnH}
	Suppose that $G$ is a discrete amenable MAP group, $A$ is an FDI C*-algebra and we have an action $\alpha: G \to \textup{Aut}(A)$. Consider the following conditions:
	\begin{itemize}
		\item[(i)] $A \rtimes G$ is RFD;
		\item[(ii)] the action $\hat{\alpha}$ of $G$ on $\beta\Prim(A)$ has a dense set of periodic points;
		\item[(iii)] $ZM(A)\rtimes_{\alpha} G$ is RFD.
	\end{itemize}	
	Then we have the following implications: (i)$\Longrightarrow$(ii)$\Longleftrightarrow$(iii).
\end{corollary}
\begin{proof}
	The implication (i)$\Longrightarrow$(ii) of Theorem \ref{th:maincrossed}, together with the comments in the beginning of this section implies that if (i) holds, then every point of $\Prim(A)$ can be approximated by points with finite orbits.
	As the image of $\beta$ is dense in $\beta \Prim(A)$, it is easy to see that if every point of $\Prim(A)$ can be approximated by a net of points with finite orbits with respect to the $G$-action, the same statement remains true for points in $\beta\Prim(A)$. This yields (ii).
	
	The equivalence (ii)$\Longleftrightarrow$(iii) follows already from Corollary \ref{TransformationGroupoid} and the identification $C(\beta \Prim(A)) \cong ZM(A)$ mentioned above.
\end{proof}

It is tempting to ask whether the implication (i)$\Longrightarrow$(ii) in the above corollary can be reversed. This is however not so clear, as formally there is no reason why in the context of the discussion before the theorem knowing that for a given action of $G$ on $X$ its `Stone-\v{C}ech extension' has a dense set of periodic points should guarantee that the same is true for the original action. To see the latter just imagine $X=\Z$ with a trivial topology and $\Z$ acting on $X$ by shifts. Then $\beta X = \{pt\}$ and the `Stone-\v{C}ech extension' becomes in fact the trivial action. Note however that
in this case $C(X) = \mathbb C$, so  $A \rtimes G$ is RFD.

\bigskip

As mentioned earlier, in some cases the property of having a dense set of periodic points for a  homeomorphism in a given class is generic, see for example \cite{DF} (or the article \cite{AHK}, or the survey in \cite{Bhishan}). Our results allow that to interpret this as saying that for certain crossed products
the RFD property is indeed generic. For any (say unital) C*-algebra $A$ equip the automorphism group $\textup{Aut}(A)$ with the topology of the point norm convergence.

Recall that given a topological manifold $X$ a probability measure $\mu$ on $X$ is said to be an \emph{OU (Oxtoby/Ulam) measure} if it is faithful, nonatomic and vanishes on the boundary of $X$.

\begin{proposition}\label{prop:generic}
	Let $X$ be a compact connected topological manifold of dimension at least 2, and let $\mu \in \textup{Prob}(X)$ be an OU (Oxtoby-Ulam) measure. Let $n \in \N$ and set $A= C(X;M_n)$; equip $A$ with a faithful trace $\omega$ given by the formula
	\[ \omega (f) = \int_X \textup{tr}(f(x)) d\mu(x), \;\;\; f \in A.\]
	Then the set $\textup{Aut}_\omega(A)$ of the automorphisms  of $A$ fixing $\omega$ contains a dense $G_\delta$-subset $Z$ such that for every $\alpha \in Z$ the crossed product $A \rtimes_\alpha \Z$ is RFD.   	
\end{proposition}

\begin{proof}
As easily checked (see for example \cite[Sections 4,5]{DanaIan}), any automorphism $\alpha \in \textup{Aut}(A)$ is given by the following formula:
 \begin{equation} \label{alphaTU} (\alpha(f)) (x)  = U_x f(Tx) U_x^*, \;\;\;\; x \in X, f \in A, \end{equation}
where $T:X \to X$ is a homeomorphism (which can be identified with the action induced by $\alpha$ on $\Prim(A)$), $(U_x)_{x\in X}$ is a family of unitaries in $M_n$  and $x \mapsto Ad_{U_x} \in \textup{Aut}(M_n)$ is a continuous map. As is customary, we will write $T= \alpha_*$ and note that the map $\textup{Aut}(A) \ni \alpha \mapsto \alpha_* \in \textup{Homeo}(X)$ is continuous and surjective. By \cite{DF} the space $\textup{Homeo}_\mu(X)$ (consisting of these homeomorphisms of $X$ which preserve $\mu$) contains a dense $G_\delta$-subset $V$ such that every $T\in V$ has a dense set of periodic orbits. Set $Z=\{\alpha \in \textup{Aut}(A): \alpha_*\in V\}$.
It is easy to check that $\alpha$ preserves $\omega$ if and only if $\alpha_*$ preserves $\mu$; thus $Z\subset \textup{Aut}_\omega(A)$, and naturally $Z$ is a $G_\delta$-set.   Given an automorphism $\alpha\in \textup{Aut}_\omega(A)$ given by  \eqref{alphaTU} we can approximate $T$ by homeomorphisms $T_i \in V$. It is then again easy to verify that the automorphisms $\alpha_i\in \textup{Aut}_\omega(A)$ given by
 \[(\alpha_i(f)) (x)  = U_x f(T_ix) U_x^*, \;\;\;\; x \in X, f \in A, \]
converge to $\alpha$; thus $Z$ is a dense $G_\delta$-set in  $\textup{Aut}_\omega(A)$. Finally Corollary \ref{cor:comm} shows that if $\alpha \in Z$ then
$A \rtimes_\alpha \Z$ is RFD.
\end{proof}

\subsection{Automatic amenability}

In Theorem \ref{th:maincrossed} above we worked with amenable actions, but similarly as in Section \ref{Sec:grpd}, the equivalent conditions stated there in fact imply that $G$ itself must be amenable.  We begin with a basic lemma.

\begin{lemma} \label{lem:basic}
Let $A$ be a C*-algebra, $a \in Z(A^{**})$ and assume that $\omega, \mu \in P(A)$ are two states with equivalent GNS representations. Then $\tilde{\omega}(a)= \tilde{\mu}(a)$, where $\tilde{\omega}, \tilde{\mu}$ denote the weak$^*$-extensions of respective states to $A^{**}$.
\end{lemma}
\begin{proof}
Let $(\pi_\omega, \Hil_\omega, \Omega_\omega)$, 	$(\pi_\mu, \Hil_\mu, \Omega_\mu)$ be the corresponding GNS triples, and let $U:\Hil_\omega \to \Hil_\mu$ be the unitary such that $\pi_\omega (\cdot) = U^* \pi_\mu(\cdot)U$.
Denote the canonical normal representations of $A^{**}$ extending $\pi_\omega$ and $\pi_\mu$ by $\widetilde{\pi_\omega}$, $\widetilde{\pi_\mu}$. Then clearly $(\widetilde{\pi_\omega}, \Hil_\omega, \Omega_\omega)$, 	$(\widetilde{\pi_\mu}, \Hil_\mu, \Omega_\mu)$ are GNS triples for states $\tilde{\omega}, \tilde{\mu}$ on $A^{**}$, and we have $\widetilde{\pi_\omega} (\cdot) = U^* \widetilde{\pi_\mu}(\cdot)U$.

We further have $\widetilde{\pi_\omega}(a) \in \widetilde{\pi_\omega}(A^{**})' =\C I_{\Hil_\omega}$; set $\lambda \in \C$ so that
$\widetilde{\pi_\omega}(a) = \lambda I_{\Hil_\omega}$. Then
\begin{align*} \tilde{\mu}(a) &= \langle \Omega_\mu, \widetilde{\pi_\mu}(a) \Omega_\mu \rangle = \langle \Omega_\mu, U \widetilde{\pi_\omega}(a)U^* \Omega_\mu \rangle = \langle \Omega_\mu, \lambda \Omega_\mu \rangle = \lambda = \langle \Omega_\omega, \lambda  \Omega_\omega \rangle \\&= \langle \Omega_\omega, \widetilde{\pi_\omega}(a)  \Omega_\omega \rangle\tilde{\omega}(a).
 \end{align*}
\end{proof}

\begin{proposition} \label{prop:amenable}
Suppose that  $G$ is a discrete  group, $A$ a C*-algebra and $\alpha$ an amenable action of $G$ on $A$.	If there exists a pure state $\omega \in P(A)$ with finite orbit in $P(A)/_{\approx}$ then $G$ itself is amenable.
\end{proposition}

\begin{proof}
Definition of amenable action \cite[D\'efinition and Theorem 3.3]{Claire} implies that there exists a $G$-invariant conditional expectation $\mathbb{E}: \ell^{\infty}(G) \overline{\otimes} Z(A^{**}) \to Z(A^{**})$.
Let $(\omega_1, \ldots, \omega_n)$ be the representatives of the orbit of $\omega \in P(A)/_{\approx}$. Define $\mu=\frac{1}{n}\sum_{k=1}^n \omega_k$. Denote the canonical weak$^*$-extensions of states on $A$ to $A^{**}$ by adorning them with tildes, so that for example $\tilde{\mu} \in S(A^{**})$ is the canonical weak$^*$-extension of $\mu$. Fix $g \in G$. Then for every $k\in\{1,\ldots,n\}$ there exists a unique $g(k)\in \{1, \ldots,n\}$ such that $\omega_k \circ \alpha_g \approx \omega_{g(k)}$.  Lemma \ref{lem:basic} implies that $\widetilde{\omega_k \circ \alpha_g} (a) = \widetilde{\omega_{g(k)}}(a)$ for every $a \in Z(A^{**})$. It is then easy to see that $\tilde{\mu}|_{Z(A^{**})}$ is a $G$-invariant state; thus the formula $f \to (\tilde{\mu} \circ \mathbb{E})(f \otimes 1_{Z(A^{**})})$ defines a $G$-invariant state on  $\ell^{\infty}(G)$, so that $G$ is amenable.

Alternatively one can deduce the result from \cite[Corollaire 4.4]{Claire}. Let $\rho:A \to B(\Hil_\omega)$ denote the GNS-representation of $\omega$. It is then a non-zero factorial representation; \cite[Corollaire 4.4]{Claire} implies that $G_\rho:=\{g \in G: g \cdot \rho \approx_{qe} \rho\}$ is amenable. But for irreducible representations quasi-equivalence coincides with unitary equivalence so the assumptions of the proposition imply that $G_\rho$ is a finite index subgroup of $G$. Hence $G$ is amenable too.
\end{proof}	

We thus obtain the following result.

\begin{corollary}\label{cor:amen}
Suppose that  $G$ is a discrete  group, $A$ a C*-algebra 	and $\alpha$ is an amenable action of $G$ on $A$ such that $A \rtimes G$ is RFD. Then $G$ is amenable.
\end{corollary}
\begin{proof}
Follows from implications (i)$\Longrightarrow$(ii) of Theorem \ref{th:maincrossed}, (ii)$\Longrightarrow$(iv) of Theorem \ref{thm:justact} and from Proposition \ref{prop:amenable}.
\end{proof}

\medskip

 \subsection{Non-amenable case}

	The assumption of amenability of (an action of) $G$ in Theorem  \ref{th:maincrossed} likely cannot be replaced by the assumption that $C^*(G)$ is RFD.  For example, let $G = F_2$ act trivially on $A = C^*(F_2)$. By \cite{Choi} $C^*(G) = C^*(F_2)$ is RFD. Hence finite-dimensional representations are dense in $\Prim(A) = \Prim(C^*(F_2))$. Since the action is trivial, all orbits in $\Prim(A)$ are 1-point sets. Thus finite-dimensional representations with finite orbits are dense in $\Prim(A)$. Since the action is trivial, $A\rtimes G = C^*(F_2\times F_2)$. The question whether this C*-algebra  is RFD is nowadays called Kirchberg's conjecture and the positive answer is known to be equivalent to a positive solution of the Connes Embedding Problem (see \cite{Taka}). Thus a negative solution to the latter proposed in \cite{CEP} suggests that $A\rtimes G $ is not RFD.

\medskip

	It would be interesting to know whether the amenability assumption can be replaced by the assumption that $C^*(G)$ is RFD in the case of commutative $A$.  Here is an interesting example to look at. Let $SL_2(\mathbb Z)$ act canonically on $\mathbb Z^2$.  Then $ C^*(\mathbb Z^2) \rtimes SL_2(\mathbb Z) =  C^*(\mathbb Z^2\rtimes SL_2(\mathbb Z))$.  It is known that $C^*(SL_2(\mathbb Z))$ is RFD. It is easy to check that periodic points are dense in $\mathbb T^2 = \hat{\mathbb Z}^2 = \Prim(C^*(\mathbb Z^2))$.
	
	\medskip
	
	{\bf Question:} Is $C^*(\mathbb Z^2\rtimes SL_2(\mathbb Z))$ RFD?
	
	\medskip

\medskip

\section{Semi-direct products of groups} \label{sec:Semidirect}

Since C*-algebras of semi-direct products of groups are particular cases of crossed products, we state here consequences of our results for the study of semi-direct products.

\subsection{RFD and MAP}

Given a discrete group $H$ we denote by $\Rep(H)$ ($\widehat H$) the set of equivalence classes of unitary representations of $H$ (irreducible unitary representations of $H$, respectively) with Fell topology.  Fell topology on $\widehat H$ coincides with Jacobson's topology (\cite{BookPropT}, p. 438).  If another group acts on $H$ by automorphisms, this also induces the action on  $\Rep(H)$ and  the action on $\widehat H$ by homeomorphisms (see Lemma \ref{homeoPrim} above). Note that we shall often identify (unitary) representations of $H$ with representations of the full group C*-algebra $C^*(H)$ or of the group ring $\C[H]$ without further comments.

\begin{corollary}\label{SemidirectProductsVersion} Suppose that $H, G$ are discrete groups, $G$ is amenable and acts on $H$ by automorphisms. Then $C^*(H \rtimes G)$ is RFD if and only if $G$ is MAP and every element of $\widehat H$ can be approximated in the Fell topology by finite-dimensional elements of $\widehat H$ whose orbits with respect to the action of $G$ are finite.	
\end{corollary}
\begin{proof}
This is equivalence (i)$\Longleftrightarrow$(ii) in Theorem \ref{th:maincrossed}.
\end{proof}

	Below we give a characterization of when a general semi-direct product is MAP. This is a consequence not of the statement of  Theorem \ref{th:maincrossed} but of its proof.

First, a lemma. It uses arguments similar to ones  Bekka and Louvet had in their proof of equivalence of MAP and RFD for amenable groups (\cite{BekkaLouvet}) and  which we already used here several times. As kindly pointed out by Bachir Bekka, it can be also deduced from more general \cite[Corollary 1.8]{BekkaInvariant}.

\begin{lemma}\label{SeparatingPointsOfGroupAlgebra} Let $\mathcal F$ be a family of finite-dimensional representations of a group $G$ that is closed under tensor products and under  direct sum with the trivial representation. If $\mathcal F$ separates points of $G$, then it separates points of the group algebra $\mathbb C G$.
\end{lemma}
\begin{proof} Let $\delta_e$ denote the standard delta-function at the neutral element of $G$  and let $\tilde \delta_e$ be its extension to a functional on $\mathbb CG$. First, we show that $\delta_e$ is a pointwise limit of traces of representations from $\mathcal F$ and therefore so is $\tilde \delta_e$.

 Let $\epsilon>0$, $g_{1}, \ldots, g_{n} \in G$, $g_{i} \neq e$ for all $i=1,
\ldots, n$. Since $\mathcal F$ separates points of $G$, we can find $\pi\in \mathcal F$ such that $\pi(g_{i}) \neq1$ for all
$i=1, \ldots, n$. Let $\chi: G \to \mathbb{C}$ be the trivial representation,
$\tilde \pi= \pi \oplus \chi$. Then
\[
|\tr (\tilde \pi(g_{i}))| = \left| \frac{(\dim\,\pi)\tr (\pi(g_{i})) +1}{\dim\,\pi+
1}\right|  < 1,
\]
since this is the absolute value of the average of numbers of absolute value not
larger than 1, not all of which are equal.

Let $\tilde \pi^{\otimes N}$ be the N-th tensor power of the representation
$\tilde \pi$. Then
\[
\tr (\tilde \pi^{\otimes N} (g_{i})) = (\tr (\tilde \pi(g_{i})))^{N} < \epsilon
\]
if $N$ is large enough. Thus
\[
|\tr (\tilde \pi^{\otimes N}(g_{i})) - \delta_{e}(g_{i})| < \epsilon
\]
for $i = 1, \ldots, n$.

As $\tilde\delta_e$ is known to be faithful (that is sends non-zero positive elements of $\mathbb C G$ to positive numbers), for any $0\lneq x \in \mathbb CG$ there is $\rho\in \mathcal F$ such that $\tr (\rho(x)) > 0$. Hence $\rho(x) \neq 0$. Thus representations from $\mathcal F$ separate positive elements of $\mathbb CG$  and hence all elements of $\mathbb CG$.
\end{proof}

The next remark follows from Remark \ref{remark:trivial2}.

\begin{remark}\label{DecompositionIrreducible}  Let $G$ act on $H$ by automorphisms, let $\rho_k$, $k=1, \ldots, N$, be irreducible finite-dimensional representations of $H$, and let $\rho = \oplus_{k=1}^N \rho_k$. Then $\rho$ has a finite orbit in $\Rep(H)$ if and only if each $\rho_k$ has a finite orbit in $\widehat H$.
\end{remark}
%\begin{proof} The "if" is clear and the "only if" follows from uniqueness of decomposition of  a finite-dimensional representation into direct sum of irreducible ones.
%\end{proof}

\begin{proposition}\label{MAPversion}  Suppose that $H, G$ are discrete groups and $G$ acts on $H$ by automorphisms.  The following are equivalent:

\begin{itemize}
		\item [(i)] $H\rtimes G$ is MAP;

\medskip

\item[(ii)]  $G$ is MAP and irrreducible finite-dimensional representations of $H$ whose orbits in $\widehat H$ are finite  separate points of $H$.
\end{itemize}
\end{proposition}
\begin{proof} $ (i) \Longrightarrow  (ii)$: Let $e\neq h \in H$. Since $H\rtimes G$ is MAP, there is a finite-dimensional representation $\pi$ of $H\rtimes G$  such that $\pi(h) \neq \mathbb 1$.  Then $\rho:=\pi|_H$ has a trivial orbit in $Rep(H)$, since
$g \cdot \rho  = \pi(g)^{-1} \rho \pi(g), $ for any $g\in G$. Decompose  $\rho = \oplus_{k=1}^N \rho_k$ as n the above remark. Then there is $1\le k\le N$ such that $\rho_k(h) \neq \mathbb 1$. By Remark \ref{DecompositionIrreducible}  the orbit of $\rho_k$ is finite.

\medskip

 $ (ii) \Longrightarrow  (i)$: Let $e\neq a \in H \rtimes G$. Then $0\neq (a-e)^*(a-e) \in \mathbb C(H\rtimes G).$ The element $a$ is of the form $a= hg$ with $h\in H, g\in G$. Since  the conditional expectation $\Phi: C^*(H \rtimes G) \to C^*(H)$ is defined on elements of $\mathbb C [H\rtimes G]$  as $$\Phi(\sum_{g\in G} h_g g) = h_e, $$ in the case $a= hg$ with $g\neq e$ we have $\Phi( (a-e)^*(a-e))  = 2e$ and in the case $a= h$ we have $\Phi( (a-e)^*(a-e)) = 2e  - h - h^{-1}$, so either way $$0 \neq \Phi( (a-e)^*(a-e)) \in \mathbb CH.$$
 It is straightforward to check that the family of finite-dimensional representations of $H$  with finite orbits  is closed under  tensor products and direct sum with the trivial representation. By Lemma  \ref{SeparatingPointsOfGroupAlgebra}  there is a finite-dimensional representation $\rho$ of $H$ with finite orbit such that $\rho(\Phi\left((a-e)^*(a-e)\right))\neq 0.$ Using Remark \ref{DecompositionIrreducible} we find an irreducible finite-dimensional representation, which we again denote by $\rho$,  whose orbit is finite and $\rho(\Phi\left((a-e)^*(a-e)\right))\neq 0.$ Then the proof  of the implication (ii)$\Longrightarrow$(i) in Theorem \ref{th:maincrossed} applies to $x:= (a-e)^*(a-e)$ and gives a finite-dimensional representation $\tilde \pi$ of $H\rtimes G$ such that $\tilde\pi(x)\neq 0$ and therefore $\tilde\pi(a)  \neq \mathbb 1.$	
\end{proof}

The last result appears also as Corollary E (i) in \cite{BekkaPreprint}.
\medskip

\subsection{FD and RF}

The property FD was introduced by  Lubotzky and Shalom in \cite{LubotzkyShalom}: a group $G$ has FD if its representations with finite image are dense in the unitary dual.

Below we state characterizations of the properties FD and RF for semi-direct products. They can be shown exactly as in the proofs for RFD and MAP properties above, once we note the following facts:

(i)  the representation constructed in Proposition \ref{prop:repconstruction} (but in the context of group representations) has finite image if and only if the building blocks have finite images;

(ii) the class of representations with finite images satisfies the assumptions of Lemma \ref{SeparatingPointsOfGroupAlgebra}.
\smallskip

Let us add more comments on the first statement, as the second is completely obvious. If the representation $\rho$ considered in Lemma \ref{lem:implement} has finite image, then
the projective representation constructed there can be also assumed to have finite image inside some $U(n)$, and then the actual representation $V$ will also have finite image: a finite group $G_1$ inside $U(n^2)$. Similarly  $\rho \otimes I$ has then finite image, say $G_2 \subset U(n^2)$. As the pair above is covariant, the resulting image  of the semidirect product under the corresponding representation will be $G_2\rtimes G_1$, a finite group inside $U(n^2)$.

\begin{proposition}\label{FDversion} Let $G$ be an amenable group and let $G$ act on $H$ by automorphisms. Then $H\rtimes G$ has FD if and only if  $G$ is RF and every element of $\widehat H$ can be approximated in the Fell topology by representations with finite images  whose orbits are finite.
\end{proposition}

\begin{proposition}\label{RFversion}  Let $G$ act on $H$ by automorphisms.  Then $H\rtimes G$ is RF if and only if $G$ is RF and elements of $\widehat H$ with finite images  whose orbits are finite separate points of $H$.
\end{proposition}
The last result appears also as Corollary E (ii) in \cite{BekkaPreprint}.

\section{Examples} \label{sec:Examples}

In this section we present several concrete examples of groups which are MAP, RF or which have RFD  full $C^*$-algebras or have FD.

\subsection{Property FD}

\begin{lemma}\label{FDfiniteorbits} Let $H$ be a finitely generated group with property FD and let $G$ act on $H$ by automorphisms. Then the elements of $\widehat H$ with finite images and finite orbits are dense in $\widehat H$.
\end{lemma}
\begin{proof} As usual, we denote by $\alpha$  an action of $G$ on $H$. Since $H$ has property (FD), its  representations factorizing through finite groups are dense in $\widehat H$.  We will show that a  representation  of $H$ factorizing through a finite group has a finite orbit in  $\widehat H$. Let then $\pi$ be a representation of $H$ such that there is a finite group $F$, a homomorphism $\gamma: H \to F$  and a representation $\beta$ of $F$ such that $\pi = \beta \circ \gamma$.  Then for any $g\in G$ and  $h\in H$

$$(\hat\alpha (g)\pi)(h) = \pi\left(\alpha (g^{-1})(h)\right) = \beta\left(\gamma(\alpha(g^{-1})(h))\right). $$ Hence $$\hat\alpha(g) \pi = \beta\circ \tilde \gamma, $$ where $\tilde \gamma = \gamma \circ \alpha(g^{-1}): H \to F$.
Since $H$ is finitely generated, there are only finitely many homomorphisms from $H$ to $F$. Therefore the orbit of $\pi$ in $\widehat H$ is finite.
\end{proof}

\begin{theorem} \label{thm:FDsemid}
	Suppose $H$ is finitely generated and $G$ is amenable. Then $H\rtimes G$ has FD if and only if $H$ has FD and $G$ is RF.
\end{theorem}
\begin{proof} Follows from Proposition \ref{FDversion} and Lemma \ref{FDfiniteorbits}.
\end{proof}

Since finitely generated free groups, surface groups and amenable RF groups have FD by \cite{LubotzkyShalom}, we obtain the following result. It generalises \cite[Th. 2.8]{LubotzkyShalom}, which says that $H\rtimes \mathbb Z$ has FD for $H$ as in the next corollary.

\begin{corollary} Let $H$ be  the free group $F_n$ or a surface group or a finitely generated amenable RF group. Then for any amenable RF group $G$, $H\rtimes G$ has FD.
\end{corollary}

We also obtain a version of the result above for other properties considered in this paper.

\begin{theorem}\label{FD} Let $H$ be a finitely generated group with the property FD (for example, the free group $F_n$ or a surface group or any finitely generated amenable RF group). Then
 \begin{itemize}
\item [(i)] for any  amenable MAP group $G$, the C*-algebra
$C^*(H\rtimes G)$ is RFD.

\item [(ii)] for any MAP group $G$, the group $H\rtimes G$ is MAP.
\end{itemize}
\end{theorem}
\begin{proof} Follows from Corollary  \ref{SemidirectProductsVersion}, Corollary \ref{MAPversion} and  Lemma  \ref{FDfiniteorbits}.
\end{proof}

\begin{remark} We could also formulate an RF version of the theorem above, but it would  not give any new information, as it is well known that semidirect product of a finitely generated RF group by an RF group is RF. Same refers to the study of the RF property for examples of Subsection \ref{subsec:T} below.
\end{remark}

Finally we would like to add one more corollary.

\begin{corollary}
	If $G$ is a  1-relator group with torsion, then $C^*(G)$ is RFD.
\end{corollary}
\begin{proof}
	The fact that every 1-relator subgroup with torsion is virtually free-by-cyclic, conjectured by Baumslag, was recently proved in \cite{KL}.
	
	Now \cite[Lemma 4.6]{BekkaLouvet} (a result saying that the property of the group $C^*$-algebra of $G$ being RFD can be deduced from the same property for a finite index subgroup of $G$) and Theorem \ref{FD} (i) yield the proof.
	
\end{proof}

\smallskip

\subsection{Property (T)} \label{subsec:T}

\begin{theorem}\label{FinManyIrreps} Let $H$ be a group  such that for each dimension $n$, $n\in \mathbb N$, $H$ has only finitely many pairwise non-unitarily equivalent  irreducible representations of dimension $n$.   Then
 \begin{itemize}
\item [(i)] if  $C^*(H)$ is RFD, then  for any  MAP amenable group   $G$ the C*-algebra   $C^*(H\rtimes G)$ is RFD;

\item [(ii)]  if $H$ is MAP, then for any MAP group $G$,   $H\rtimes G$ is MAP.

\end{itemize}
\end{theorem}
\begin{proof} The induced action of $G$ on $\widehat H$ preserves the dimension. Therefore each irreducible finite-dimensional representation of $H$ has a  finite orbit in $\widehat H$.  The statements follow now from Corollary \ref{SemidirectProductsVersion} and Corollary \ref{MAPversion}.
\end{proof}

\begin{corollary} \label{cor:T} Let $H$ be a group with  property (T). Then
 \begin{itemize}
\item [(i)]  if $H$ is MAP, then for any MAP group $G$,  $H\rtimes G$ is MAP;

\item [(ii)] if $C^*(H)$  is RFD, then for any MAP amenable group $G$,  $C^*(H\rtimes G)$ is RFD.
\end{itemize}
\end{corollary}
\begin{proof}
It follows from Theorem \ref{FinManyIrreps} and  the fact that a property (T) group has only finitely many pairwisely non-unitarily equivalent  irreducible representations of each dimension \cite[Theorem 2.6]{Wang}.
\end{proof}

It is an open question of whether a  property (T) group can have RFD full C*-algebra.

\bigskip

\subsection{Free product actions} \label{subsec:free}

Exel and Loring's characterization of RFD in terms of approximation of representations was their main tool for proving that RFD passes to free products (\cite{ExelLoring}). Similarly our dynamical version of Exel-Loring result allows to deal with free products of actions.

Let $A$ and $B$ be two C*-algebras and let $i_1: A \to A\ast B$ and $i_2: B \to A\ast B$ denote the canonical embeddings. Suppose $\alpha$ is an action of $G$ on $A$ and $\beta$ is an action of $G$ on $B$.  Then $G$ acts on $A\ast B$ with the action $\gamma$ (which we also will denote by $\alpha \ast \beta$) defined as follows.
For $g\in G$, we define $\gamma_1(g): A \to A\ast B$ and $\gamma_2(g): B \to A\ast B$ by
$$\gamma_1(g)(a) = i_1(\alpha(g)(a)), \; \gamma_2(g)(b) = i_2(\beta(g)(b)),$$
$a\in A, b\in B$. Then $\gamma_i(g)$ is a $\ast$-homomorphism, for $i=1, 2, g\in G$. By the universal property of free products, for each $g\in G$ we have a $\ast$-homomorphism $\gamma(g): A \ast B \to A\ast B$ which is obviously invertible. It is straightforward to check that $\gamma: G \to \textup{Aut} (A\ast B)$ is a group homomorphism.
%{\color{blue} Namely one checks that $\gamma(g_1g_2)$ and $\gamma(g_1)\circ \gamma(g_2)$ coincide on elements from $i_1(A)$ and $i_2(B)$ and since both of them are $\ast$-homomorphisms, then coincide on the whole $A\ast B$.}

\medskip

The proof of the following theorem goes along the same lines as the proof of Exel and Loring that RFD passes to free products. Free products can be both unital and non-unital.

\begin{theorem}\label{FreeProductAction} Let $G$ be amenable. Suppose $A\rtimes_{\alpha} G$ and $B\rtimes_{\beta} G$ are RFD. Then the crossed product $(A\ast B)\rtimes_{\alpha\ast \beta} G$ is RFD.
\end{theorem}
\begin{proof} By Theorem \ref{thm:justact} and \ref{th:maincrossed} it is sufficient to show that any representation $\rho$ of $A\ast B$ is a corner of a representation that can be  SOT-approximated by finite-dimensional representations with finite orbits. We can assume that $\rho$ is infinite-dimensional, otherwise we can replace it by $\rho^{\oplus \infty}$.  We write $\rho = \rho_A \ast \rho_B$, where $\rho_A = \rho|_{i_1(A)}, \rho_B = \rho|_{i_2(B)}.$   Since $A\rtimes_{\alpha} G$ and $B\rtimes_{\beta} G$ are RFD, by Theorems \ref{thm:justact} and \ref{th:maincrossed} the representation $\rho_A$ ($\rho_B$, respectively)
is a corner of a representation $\tilde \rho_A$ ($\tilde\rho_B$, respectively) that can be SOT-approximated by finite-dimensional representations with finite orbits. It might happen that $\tilde \rho_A, \tilde \rho_B$ live on spaces of different cardinality, say
$$m:= \dim \tilde \rho_A  > k:= \dim \tilde \rho_B.$$
Since for infinite cardinals one has $km = \max\{k, m\},$ we obtain
$$
\dim \tilde\rho_B^{(\oplus m)} = m \dim \tilde\rho_B = \max\{m, k\} = m.$$
Replacing $\tilde\rho_B$ by $\tilde\rho_B^{(\oplus m)}$, we now can assume that $\tilde \rho_A$ and $\tilde \rho_B$ live on the same space $\Hil$, $\rho_A, \rho_B$ are their corners, and
$$\tilde\rho_A = SOT-\lim \rho_n^A, \; \tilde \rho_B = SOT-\lim \rho_n^B,$$
where $\rho_n^A, \rho_n^B$ are finite-dimensional with finite orbits. A priori $\rho_n^A$ and $ \rho_n^B$ live on different subspaces of $\Hil$, say $E_n$ and $F_n$.

In the non-unital case, by adding zero summands to  $\rho_n^A, \rho_n^B$ we think of them as representations on $\Hil$ and immediately conclude that
$$\tilde\rho_A \ast \tilde\rho_B = SOT-\lim \rho_n^A \ast \rho_n^B.$$

In the unital case, let $\Kil_n$ be a finite-dimensional subspace of $\Hil$ containing both $E_n$ and $F_n$ and with dimension a common multiple of $\dim E_n$ and $\dim F_n$. Let $\bar\rho_n^A$ ($\bar\rho_n^B$, respectively) be the representation of $A$ ($B$, respectively) on $K_n$ obtained by the appropriate multiple of $\rho_n^A$ ($\rho_n^B$, respectively). Then, by \cite[Lemma 3.1]{ExelLoring},  we still have \begin{equation}\label{1} \rho_A = SOT-\lim \bar\rho_n^A, \; \rho_B = SOT-\lim \bar\rho_n^B\end{equation} and $\bar\rho_n^A$, $\bar\rho_n^B$ have finite orbits and live on the same finite-dimensional subspace of $\Hil$. It implies
$$\rho = SOT-\lim \bar\rho_n^A \ast \bar\rho_n^B.$$

It remains to notice that $\rho_n^A \ast \rho_n^B$ ($\bar\rho_n^A \ast \bar\rho_n^B$, respectively) has finite orbit.
Indeed it is straightforward to check that if we have representations, say $\pi_A$ of $A$ and $\pi_B$ of $B$, then $\hat{\alpha\ast\beta}(g)(\pi_A\ast \pi_B) = \hat \alpha(g)(\pi_A)\ast \hat\beta(g)(\pi_B)$.
\end{proof}

We will now present an example of application of the above result.

Let $m \in \N, m \geq 2$. Recall that the \emph{Baumslag- Solitar group} $BS(1, m)= \langle y^{-1}ay= a^m\rangle$ can be written as semidirect product
$$BS(1, m) = \mathbb Z[1/m]\rtimes _{\alpha} \mathbb Z,$$ where $\alpha$ is the multiplication by $m$ and the isomorphism is given by
$$a\mapsto \frac{1}{m}\in \mathbb Z[1/m], \; y \mapsto 1\in \mathbb Z.$$

\begin{proposition}
Let $m, n \in \N$. The group  $BS(1, m) \ast_{\mathbb Z} BS(1, n)$ (where the amalgamation takes place over the copies of $\Z$ given by `$y$-elements') is RF, and $C^*(BS(1, m) \ast_{\mathbb Z} BS(1, n))$ is RFD.
\end{proposition}

\begin{proof}
Since $BS(1, m)$ is amenable and RF, $C^*(\mathbb Z[1/m]\rtimes _{\alpha} \mathbb Z)$ is RFD. By Theorem \ref{FreeProductAction} $C^*((\mathbb Z[1/m]\ast \mathbb Z[1/n])\rtimes \mathbb Z$ is RFD. Since
$$(\mathbb Z[1/m]\ast \mathbb Z[1/n])\rtimes \mathbb Z = \langle y^{-1}ay = a^m, y^{-1}by = b^m\rangle = BS(1, m)\ast_{\mathbb Z} BS(1,n),$$
we conclude  that
	$C^*(BS(1, m) \ast_{\mathbb Z} BS(1, n))$ is RFD. Since $BS(1, m) \ast_{\mathbb Z} BS(1, n)$ is finitely generated, this implies also that it is RF.
\end{proof}
%{\color{blue} Indeed $\mathbb Z[1/m]\rtimes _{\alpha} \mathbb Z$ is generated by elements $y$ and $\frac{k}{m^n}$ with the relation $y^{-1} \frac{k}{m^n} y = \alpha(\frac{k}{m^n}) = \frac{k}{m^{n-1}}$. Let $a= 1/m$. Then $y^{-1}ay = a^m$, $a, y$ generate the semidirect product and all the relations follow from this one.}

%{\color{blue}  I think even the latter fact was not known before. There is a list of sufficient conditions on when amalgamted free product of groups is RF, but they do not hold for this case. }

\bigskip

\subsection{Semidirect products of  $\left(\mathbb Z[\frac{1}{p}]\right)^n$}

Let $p\in \mathbb N$ be prime.
%The abelian group $\mathbb Z[\frac{1}{p}]$, despite of being non-finitely generated, admits at most finitely many homomorphisms to any given finite group.
We will need the following proposition which in the case $p=2$ can be found in  \cite[Lemma 5.6]{ES}. The proof for arbitrary $p$ is  almost the same, only with a small change. We write it here for the reader convenience.

\begin{proposition}\label{prime} Let $S\leq \mathbb Z[\frac{1}{p}]$ be non-cyclic.  Let $k_0=\min\{ k\in \Z^+: k\in S \}.$  Then $S=k_0\cdot \mathbb Z[\frac{1}{p}] = \langle  \frac{k_0}{p^n}:n=0,1,... \rangle.$
\end{proposition}
\begin{proof} Let $S\leq \mathbb Z[\frac{1}{p}]$ and suppose that $S$ has a least positive element $x$. Since $S$ is not cyclic, there is a $y\in S$ with $y>0$ such that $xk\neq y$ for all $k\in\Z^+.$  Hence there is a $k\in \Z^+$ such that $kx<y<(k+1)x.$  Then $0<y-kx<x$ and $y-kx\in S$, which is a contradiction.  We deduce that $S$ has no least positive element.

For $n\geq 0$ define $x_n\in S$ and $k_n\in\Z^+$ such that
\begin{equation*}
x_n:=\frac{k_n}{p^n} = \min\{  x\in S\cap \langle p^{-n} \rangle:x>0 \}.
\end{equation*}
Then $x_n\to 0 .$ Hence there is some index $n\in \N$ such that $\frac{k_{n+1}}{p^{n+1}}<\frac{k_n}{p^n}.$  If $k_{n+1}=pk$ for some integer $k$, then $\frac{k_{n+1}}{p^{n+1}}=\frac{pk}{p^{n+1}}=\frac{k}{p^n}<x_n$, a contradiction. Therefore $k_{n+1}$ is not divisible by $p$.
For $0\leq i<n+1$ we have
\begin{equation*}
\frac{k_{n+1}}{p^i}=x_{n+1}p^{n+1-i}\in S\cap \langle p^{-i} \rangle
\end{equation*}
Hence $k_i|k_{n+1}$ and in particular $k_i$ is not divisible by $p$.  Since $x_n\rightarrow 0$ we can find pairs $x_n>x_{n+1}$ with arbitrarily large $n$  and then repeat the above argument to obtain that $k_m$ is not divisible by $p$ for all $m$.

We now show that $k_i=k_{i+1}$ for all $i\geq 0.$
By the previous paragraph we have $k_{i+1}=k_ij$ for some $j.$  Suppose, for the sake of contradiction, that $j\neq 1$.  Since $k_{i+1}$ is not divisible by $p$, $j$ is also not divisible by $p$.  Then, since $p$ is prime, $p$ is not divisible by $j$. Therefore there is $l\in \mathbb N$ such that
\begin{equation}\label{prime1} jl > p,
\end{equation}
\begin{equation}\label{prime2} j(l-1) < p.
\end{equation}
Since $\frac{jk_i}{p^{i+1}} = x_{i+1}\in S$ and $\frac{pk_i}{p^{i+1}} = x_i\in S$, using (\ref{prime2}) we obtain

$$0<\frac{(p-(l-1)j)k_i}{p^{i+1}} \in S.$$
This implies that
$$(p-(l-1)j)k_i \ge k_{i+1} = jk_i$$ whence  $p - (l-1)j \ge j$ and we obtain $p\ge lj$. This contradicts (\ref{prime1}).

Thus let $k_0: = k_i$. We have shown that $k_0\cdot \mathbb Z[\frac{1}{p}] \subseteq S$. Suppose $S\ni x \notin k_0\cdot \mathbb Z[\frac{1}{p}]$. Then $x = \frac{mk_0+r}{p^i}$, for some $m\in \mathbb Z$, $i, r\in \mathbb N$ with $r<k_0$. Then
$$\frac{r}{p^i} = x - m \frac{k_0}{p^i} \in S $$ and $$\frac{r}{p^i}<\frac{k_0}{p^i} = \frac{k_i}{p^i}.$$ Contradiction.
\end{proof}

\begin{corollary}\label{CorPrime} Let $p$ be a prime.  For any finite group $F$ there are at most finitely many homomorphisms from $\mathbb Z[\frac{1}{p}]$ to $F$.
\end{corollary}
\begin{proof}  Since  $\mathbb Z[\frac{1}{p}]$ is abelian, it is suffices to consider only a finite cyclic $F$. By Proposition \ref{prime}, for each $k_0$  there is only one subgroup of index $k_0$ in  $\mathbb Z[\frac{1}{p}]$. Therefore all homomorphisms from $\mathbb Z[\frac{1}{p}]$ to $F$ have the same kernel. Let $f:\mathbb Z[\frac{1}{p}] \to F$  be any of them. Then any other can be written as $\alpha \circ f$, where $\alpha$ is an automorphism of $F$. The result follows.
\end{proof}

\medskip

\noindent Being amenable RF, $ \mathbb Z[\frac{1}{p}]$  has the property FD (\cite[Corollary 2.5]{LubotzkyShalom}).  Taking into account Corollaries \ref{CorPrime},  \ref{SemidirectProductsVersion}, \ref{MAPversion} and \ref{RFversion},  the same arguments as in Lemma \ref{FDfiniteorbits} show the following result.

\medskip

%{\color{blue} I intentionally want to do it for $\mathbb Z[\frac{1}{p}]$ and not just $\mathbb Z$ because for $\mathbb Z$ the result would follow from the known fact that if $H$ is finitely generated RF and G is RF, then $H\rtimes G$ is RF.}

\begin{theorem}\label{Z1p} Let  $H$ be $\mathbb Z[\frac{1}{p}]$ or a direct product of  finitely many copies of $\mathbb Z[\frac{1}{p}]$. Then
 \begin{itemize}
\item [(i)]  for any MAP group $G$,  $H\rtimes G$ is MAP;

\item [(ii)]  for any RF group $G$,  $H\rtimes G$ is RF;

\item [(iii)] for any MAP amenable group $G$,  $C^*(H\rtimes G)$ is RFD.

\end{itemize}

\end{theorem}

\medskip

The following lemma is undoubtedly known to experts but it is hard to find a reference, so we give a proof here.

\begin{lemma} \label{RingHomomorphisms}  Ring homomorphisms from  $\mathbb Z[\frac{1}{p}]$ to finite rings separate points of   $\mathbb Z[\frac{1}{p}]$.
\end{lemma}
\begin{proof} Let $q$ be a number that is mutually prime with $p$. Let $\mathbb Z_q$ denote the ring of remainders for division by $q$. For $x\in \mathbb Z$ we denote  by $[x]$ its image in $\mathbb Z_q$. Let $[x]^{-1}$ denote the multiplicative inverse of $[x]$ (such exists if $x$ is mutually prime with $p$), that is $[x]^{-1} [x] = [1].$  Define $h_q: \mathbb Z[\frac{1}{p}] \to \mathbb Z_q$ by

$$h_q\left(\frac{k}{p^m}\right) = [p]^{-m} [k],$$
for any $m, k\in \mathbb Z, m\ge 0.$ It is straightforward to check that $h_q$ is a ring homomorphism. Now let $k\neq 0$. Since the set of numbers that are mutually prime with $p$ is unbounded, there exists $q$ that is mutually prime with $p$ and does not divide $k$. Then $h_q(\frac{k}{p^m}) \neq 0$.
\end{proof}

Let $n\in \mathbb N$. Consider the canonical action of  $SL_n(\mathbb Z[\frac{1}{p}])$ on $\left(\mathbb Z[\frac{1}{p}]\right)^n$.

\begin{corollary}\label{SLnZ}  For any $n\in \mathbb N$,   $\left(\mathbb Z[\frac{1}{p}]\right)^n\rtimes SL_n(\mathbb Z[\frac{1}{p}])$
is RF.
\end{corollary}
\begin{proof} By Theorem \ref{Z1p} we only need to check that  $SL_n(\mathbb Z[\frac{1}{p}])$ is RF.
By Lemma \ref{RingHomomorphisms} ring homomorphisms from  $\mathbb Z[\frac{1}{p}]$ to finite rings  separate points. Applying these homomorphisms   entrywise to matrices from $SL_n(\mathbb Z[\frac{1}{p}])$ we see that $SL_n(\mathbb Z[\frac{1}{p}])$ is RF.
\end{proof}

\medskip

There is a well-known fact that a semidirect product of a finitely generated  RF group with  an RF group is RF. We note that this fact does not imply the result above as $\left(\mathbb Z[\frac{1}{p}]\right)^n$ is not finitely generated.

\bigskip

\subsection{Actions of $\mathbb Q$}

Let  $\mathbb Q$  be the group of all rational numbers with multiplication given by the addition. In the following proposition $\mathbb Q$ can be replaced by any other group that does not admit any  non-trivial homomorphisms to finite groups.

\begin{proposition}\label{Q} For any non-trivial action of $\mathbb Q$ on any locally compact space $X$, the crossed product $C_0(X)\rtimes \mathbb Q$ is not RFD.
\end{proposition}
\begin{proof} A point $x\in X$ has  a finite orbit if and only if its stabilizer  $St_x$ has finite index. Since $\mathbb Q$ does not have non-trivial subgroups of finite index, a point with a finite orbit must be a fixed point. If the set of fixed points is dense in $X$, then all points in $X$ are fixed. Since our action is non-trivial, this implies that the set of points with finite orbits is not dense in $X$. By Corollary \ref{TransformationGroupoid},  $C_0(X)\rtimes \mathbb Q$ is not RFD.
\end{proof}

\medskip

This can be used to show that in Corollary \ref{SLnZ} one cannot replace $\mathbb Z[\frac{1}{p}]$ by  $\mathbb Q$. The latter fact follows however already from the fact that  $SL_n(\mathbb Q)$ is not  MAP   for any $n>1$ (see for example \cite[Proposition 4.C.20]{BookUnitaryDual}).

\bigskip

\subsection{Hall's examples}

The same arguments as for Theorem \ref{Z1p} apply to some other semidirect products. For example, in \cite[p. 349]{Hall}  P.
Hall defined the following action of $\mathbb Z$ on  the Heisenberg group $\mathbb H_3\left(\mathbb Z[\frac{1}{2}]\right)$:

\begin{equation*}
\alpha\left( \left[ \begin{array}{ccc} 1 & x & z\\ 0 & 1 & y\\ 0 & 0 & 1 \end{array}  \right]  \right)=\left[ \begin{array}{ccc} 1 & 2x & z\\ 0 & 1 & \frac{y}{2}\\ 0 & 0 & 1 \end{array}  \right]
\end{equation*}

In \cite[Th. 5.23]{ES} it was proved that the corresponding semidirect product $\mathbb H_3\left(\mathbb Z[\frac{1}{2}]\right)\rtimes \mathbb Z$ is Hilbert-Schmidt stable which is a stronger property than MAP (\cite[Theorem 3 + Proposition 3]{HS}, \cite[page 2]{ES}). However the MAP, even RF,  property of this group can also be obtained by arguments  as above.

\begin{lemma}   Let $p$ be a prime.  For any finite group $F$ there are at most finitely many homomorphisms from $\mathbb H_3\left(\mathbb Z[\frac{1}{p}]\right)$ to $F$.
\end{lemma}
\begin{proof} The group $\mathbb H_3\left(\mathbb Z[\frac{1}{p}]\right)$ is generated by three subgroups

$$\left\{ \left[ \begin{array}{ccc} 1 & \ast & 0 \\ 0 & 1 & 0\\ 0 & 0 & 1 \end{array}  \right]\right\}, \; \left\{ \left[ \begin{array}{ccc} 1 & 0 & 0 \\ 0 & 1 & \ast\\ 0 & 0 & 1 \end{array}  \right]\right\}, \; \left\{ \left[ \begin{array}{ccc} 1 & 0 & \ast \\ 0 & 1 & 0\\ 0 & 0 & 1 \end{array}  \right]\right\} $$
each of which is isomorphic to $\mathbb Z[\frac{1}{p}]$. The result follows now from Corollary \ref{CorPrime}.

\end{proof}

\noindent Taking into account this lemma, the same arguments as  for Theorem \ref{Z1p} imply the result below.

\begin{proposition} Consider the action of $\mathbb Z$ on  the Heisenberg group $\mathbb H_3\left(\mathbb Z[\frac{1}{p}]\right)$ defined by
\begin{equation*}
\alpha\left( \left[ \begin{array}{ccc} 1 & x & z\\ 0 & 1 & y\\ 0 & 0 & 1 \end{array}  \right]  \right)=\left[ \begin{array}{ccc} 1 & px & z\\ 0 & 1 & \frac{y}{p}\\ 0 & 0 & 1 \end{array}  \right].
\end{equation*}
 Then the corresponding semidirect product $\mathbb H_3\left(\mathbb Z[\frac{1}{p}]\right) \rtimes \mathbb Z$ is RF.
\end{proposition}

\bigskip

\subsection{Wreath products}
Let $H$ and $G$ be groups and let $G$  act on itself by left multiplication. This action extends to an action of $G$ on $\oplus_{g\in G} H$ by defining
$$g_0  \left(h_g\right)_{g\in G} := \left(h_{g_0^{-1}g}\right)_{g\in G}.$$
Recall that the regular wreath product, denoted by $H \wr G$, is the semidirect product
$$H \wr G:= \left(\oplus_{g\in G} H\right) \rtimes G.$$
There is a well-known characterization of when wreath products are RF, due to Gruenberg.

\begin{theorem}\cite{Grunberg} Let $H,G$ be discrete groups. The wreath product $H \wr G$ is RF if and only if either $H$ is RF and $G$ is finite,  or $H$ is abelian RF and $G$ is RF.
\end{theorem}

Below we give a characterization of when wreath products are MAP. Our methods are based on our characterization of MAP semidirect products and seem to be completely different from  \cite{Grunberg}. Note that this result was independently obtained in \cite[Corollary G]{BekkaPreprint}.

\begin{theorem}\label{WreathMAP}Let $H,G$ be discrete groups. Then $H \wr G$ is MAP if and only if either $H$ is MAP and $G$ is finite, or $H$ is trivial and $G$ is MAP,   or $H$ is abelian and $G$ is RF.
\end{theorem}

The proof will follow from lemmas below. For brevity, we will write $\oplus H$   for  $\oplus_{g\in G} H$ and $(h_g)$  for $(h_g)_{g\in G}$ Given $g' \in G$, $h \in H$ we will also write $x_{g',h}$ for the element of $\oplus H$ given by  \[x_{g',h}(g) = \begin{cases}h & \textup{ if }  g=g', \\ e &  \textup{ if }  g\neq g'.\end{cases}\]

\medskip

Suppose that for each $g \in G$ we are given $\pi_g: H \to \mathbb T$, a one-dimensional representation of $H$. Define a one-dimensional representation $\prod\pi_g$ of $\oplus H$ by
$$\prod\pi_g((h_g)) = \prod_g \pi_g(h_g),$$
for any $ (h_g) \in \oplus H$. Since $ h_g\neq e_H$ only for finitely many $g$, the right-hand side is well-defined.
Since the multiplication of complex numbers is commutative, $\prod\pi_g: \oplus H \to \mathbb T$ is indeed a homomorphism. Moreover, any one-dimensional representation of $\oplus H$ arises in this way, which is easy to check.

\begin{lemma}\label{Wreath1} Let $F$ be a finite group and $\gamma: G \to F$  a homomorphism. For each $i\in F$ let $\rho_i: H \to \mathbb T$ be a one-dimensional representation of $H$ and let $\pi_g = \rho_{\gamma(g)}$.
Then $\prod\pi_g$ has   finite orbit in $\Rep(\oplus H)$. Moreover, any one-dimensional representation of $\oplus H$ with finite orbit arises in this way.
\end{lemma}
\begin{proof}  Let $g_0\in G$. For any $(h_g) \in \oplus H$ we have
\begin{multline*}\hat\alpha(g_0)(\prod\pi_g) ((h_g)) = \prod\pi_g\left((h_{{g_0}^{-1}g})\right) = \prod_g \pi_g(h_{{g_0}^{-1}g}) =  \prod_g \pi_{g_0g}(h_g) = \prod\pi_{g_0g} ((h_g)), \end{multline*}
where the third equality follows from a straightforward re-indexing. Therefore
\begin{equation}\label{eq:Wreath1}\hat\alpha(g_0)\left(\prod\pi_g\right)   = \prod\pi_{g_0g}.\end{equation}
We write $G$ as disjoint union of the left cosets $$G = g_1 \Ker\, \gamma \cup \ldots \cup g_N \Ker\, \gamma.$$ Then $g_0 \in g_{k_0} \Ker\, \gamma, $ for some $k_0\le N$, and
$$\pi_{g_0g}  = \rho_{\gamma(g_0g)} = \rho_{\gamma(g_{k_0})\gamma(g)}.$$ Then by \eqref{eq:Wreath1} $\prod\pi_g$ has   finite orbit.

To show the last statement, suppose that a one-dimensional representation $\pi$ of $\oplus H$ has finite orbit. Then the stabilizer $St_{\pi}$ of  $\pi$    is a finite index subgroup of $G$. Let
$\gamma: G \to G/St_{\pi}$ be the corresponding quotient homomorphism and let $g_1, \ldots, g_N$ be representatives of right cosets for $St_{\pi}$. For each $i\le N$ we define a one-dimensional representation $\rho_{\gamma(g_i)}$ of $H$ by
$$\rho_{\gamma(g_i)}(h) = \pi(x_{g_i, h}), \;\;\;h \in H.$$
Further for each $g \in G$ let $\pi_g: = \rho_{\gamma(g)}$. Then, by writing $g\in G$ as $g = g_0g_k$, for some $g_0\in St_{\pi}$, $k\le N$, we obtain
$$\pi_g(h) = \rho_{\gamma(g)}(h) = \rho_{\gamma(g_k)}(h) = \pi(x_{g_k, h}) = \left(\hat\alpha(g_0)\pi\right)(x_{g_0g_k, h}) = \pi(x_{g, h}),$$
for any $h\in H$.  Thus $\prod \pi_g = \pi.$
\end{proof}

\begin{lemma}\label{Wreath2} If $H$ is abelian and $G$ is RF, then $H\wr G$ is MAP.
\end{lemma}
\begin{proof} By  Corollary \ref{MAPversion} we only need to prove that finite-dimensional representations of $\oplus H$ with finite orbits separate points of $\oplus H$.  Let $e_{\oplus H} \neq (h_g) \in \oplus H$. Then there is $g_0\in G$ such that $h_{g_0}\neq e_H$.  Since $H$ is abelian, there is a one-dimensional representation $\rho$ of $H$ such that $\rho(h_{g_0})\neq 1$. Let $$S = \{g\in G \;|\; h_g \neq e_H\}.$$ Since $G$ is RF, there a finite group $F$ and a homomorphism $\gamma: G \to F$ such that $\gamma(s_1) \neq \gamma(s_2), $ for any $s_1\neq s_2$, $s_1, s_2 \in S$.   Then  for any $S \ni s\neq g_0$ we can define  $\rho_{\gamma(s)} $ to be the trivial representation and let
$ \rho_{\gamma(g_0)} = \rho$.  Let $\pi_g$, $g\in G$, be as in Lemma \ref{Wreath1}. Then, by  Lemma \ref{Wreath1}, $\prod\pi_g$ has   finite orbit in $\Rep(\oplus H)$ and we have
$$ \prod\pi_g ((h_g))  = \prod_g \pi_g(h_g) = \prod_g\rho_{\gamma(g)}(h_g) = \rho(h_{g_0}) \neq 1.$$
\end{proof}

\begin{lemma}\label{Wreath3} If $H\wr G$ is MAP and $G$ is infinite, then $H$ is abelian.
\end{lemma}
\begin{proof} For each $g \in G$ define
$$\Sigma_g = \{x_{g, h}\;|\; h\in H\}.$$
  Suppose for the sake of contradiction that $H$ is not abelian. Then there are $h_1, h_2 \in H$ such that $h_1h_2 \neq h_2h_1$. Then
$x_{e_G, h_1}x_{e_G, h_2}\neq x_{e_G, h_2}x_{e_G, h_1}$. Since $H\wr G$ is MAP, by Corollary \ref{MAPversion} there is a finite-dimensional representation $\pi$ of $\oplus H$ with finite orbit such that
\begin{equation}\label{Wreath3.1} \pi(x_{e_G, h_1}x_{e_G, h_2})\neq \pi(x_{e_G, h_2}x_{e_G, h_1}).\end{equation}
Then $N = \dim \pi \ge 2$. Let $\mathcal A_g \subseteq M_N$ be the C*-algebra generated by  $\pi(\Sigma_g)$. Then by (\ref{Wreath3.1}) $\mathcal A_{e_G}$ is a noncommutative C*-algebra of matrices and therefore is isomorphic to direct sum of full matrix algebras with at least one summand being of dimension $m> 1$.

Since $G$ is infinite and $\pi$ has finite orbit, the stabilizer of $[\pi] \in \Rep(G)/_{\approx}$ is infinite. Thus there exist infinitely many distinct elements $g_1 = e_G, g_2, \ldots$ such that for each $i \in \N$ we have
\begin{equation}\label{Wreath3.2} \hat\alpha(g_i)\pi = u_i^*\pi u_i,\end{equation}
for some unitary $u_i\in M_N$. Therefore
$$\mathcal A_{g_i} = \{u_iau_i^*\;|\; a\in \mathcal A_{e_G}\} \cong \mathcal A_{e_G}$$ and  therefore this algebra is isomorphic to a direct sum of full matrix algebras with at least one summand being of dimension $m> 1$. Let  $\phi_i: M_m \to \mathcal A_{g_i}$ be the corresponding embedding.

 Since $\Sigma_{g_i}$ commutes with $\Sigma_{g_j}$ when $i\neq j$, we have
\begin{equation}\label{Wreath3.3} \mathcal A_{g_i} \;\text{commutes with}\;  \mathcal A_{g_i}, \; \text{when}\; i\neq j.\end{equation}
Let $k\in \mathbb N$ be such that $m^k > N$. By (\ref{Wreath3.3}) we can define a $\ast$-homomorphism $f:  M_m^{\otimes k} \to M_N$ by the formula
$$f(T_1\otimes \ldots \otimes T_k) = \phi_1(T_1)\ldots \phi_k(T_k).$$ However there can be no such homomorphism since $M_m^{\otimes k}  \cong M_{m^k}$ is simple. Contradiction.
\end{proof}

\medskip

\begin{lemma}\label{Wreath4} If $H\wr G$ is MAP and $H$ is non-trivial, then $G$ is RF.
\end{lemma}
\begin{proof} We can assume that $G$ is infinite, as finite groups are RF. Then by Lemma \ref{Wreath3} $H$ is abelian. Let $e_G \neq g_0 \in G$. Let $e_H \neq h_0 \in H$ and consider the element $ y: = (h_g) \in \oplus H$ defined by
$$h_g = \begin{cases} h_0 &\textup{ if } g= g_0, \\ h_0^{-1} &\textup{ if } g= e_G, \\ e_H &\textup{ otherwise. }\end{cases} $$
 Then $y\neq e_{\oplus H}$,  and by Corollary \ref{MAPversion} there is an irreducible, hence 1-dimensional, representation $\pi$ of $\oplus H$ with finite orbit such that $\pi(y) \neq 1$.  By Lemma \ref{Wreath1} there is a finite group $F$, a homomorphism $\gamma: G \to F$ and, for each $i\in F$, a 1-dimensional representation $\rho_i$ of $H$ such that $\pi  = \prod \rho_{\gamma(g)}.$
Then
$$1\neq \pi(y) = \rho_{\gamma(g_0)}(h_0)\rho_{e_F}(h_0^{-1}), $$ hence
$$\rho_{\gamma(g_0)}(h_0) \neq \rho_{e_F}(h_0), $$ hence
$\rho_{\gamma(g_0)} \neq \rho_{e_F}$, hence $\gamma(g_0) \neq e_F$. Since $g_0$ was an arbitrary non-trivial element of $G$, we proved that $G$ is RF.
\end{proof}

\medskip

{\it Proof of Theorem \ref{WreathMAP}.}  Suppose $H\wr G$ is MAP.  Then $H$ and $G$ are MAP and either 1) $H$ is trivial or 2) $H$ is non-trivial and $G$ is finite
or 3) $H$ is non-trivial and $G$ is infinite, in which case $H$ is abelian by Lemma \ref{Wreath3} and $G$ is RF by Lemma \ref{Wreath4}.

For the converse note that if $G$ is finite, then all orbits in $\widehat{\oplus H}$ are finite, so if additionally $H$ is MAP, then $H\wr G$ is MAP by Corollary \ref{MAPversion}. If $H$ is trivial  and $G$ is MAP, then $H\wr G = G$ is MAP. If $H$ is abelian and $G$ is RF, then $H\wr G$ is MAP by Lemma \ref{Wreath2}. \qed

\medskip

\begin{remark} In the same way as above we can obtain Gruenberg's  characterization of the RF property for wreath products.
\end{remark}

\medskip

\subsection{Full G-shift}

Let as always G be a countable group  and let A be a finite set with at least two elements. The set $$A^G = \{x: G\to A\}$$
equipped with the product topology is a compact Hausdorff totally disconnected space space  \cite{Ceccherini} (e.g.\ when $A$ is a 2-point set, $A^G$ is the standard presentation of the Cantor set).     The {\it full G-shift} is the action of $G$ on $A^G$ defined by
$$gx(h) = x(g^{-1}h).$$
 By \cite[Th. 2.7.1]{Ceccherini} a group $G$ is RF if and only the set of points of $A^G$ which have finite orbits for this action  is dense in $A^G$ . Therefore by Theorem \ref{th:maincrossed} (or by Corollary  \ref{TransformationGroupoid}) we obtain
that for an amenable group $G$ and a finite set $A$ with at least two elements the corresponding crossed product C*-algebra $C(A^G)\rtimes G$ is RFD if and only if $G$ is RF.

\medskip

\noindent {\bf Acknowledgments. }
A.S.\ was  partially supported by the National Science Center (NCN) grant no. 2020/39/I/ST1/01566. T.S. was partially supported by a grant from the Swedish Research Council. We thank Bachir Bekka for his comments on the first draft of this work, which led to several presentation improvements. We also acknowledge a very detailed report of the referee, which helped us improve the exposition in several places.

%\medskip

%\noindent {\bf Data availability}
%This manuscript does not possess any supplementary or affiliated data.

%\medskip

%\noindent {\bf Conflict of interest.}
%On behalf of all authors, the corresponding author states that there is no conflict of interest.


\begin{thebibliography}{99999999}
	
	
	
	\bibitem[AHK]{AHK} E.\,Akin, M.\,Hurley, and J.A.\,Kennedy,
	Dynamics of topologically generic homeomorphisms.
	Mem. Amer. Math. Soc. 164 (2003), no. 783, viii+130 pp.

\bibitem[ANT]{ANT} V. Alekseev, T. Netzer, and A. Thom, Quadratic modules, C*-algebras,
and free convexity, Trans. Amer. Math. Soc. 372 (2019), no. 11, 7525–7539.

	
	
	\bibitem[AD]{Claire} C.\,Anantharaman-Delaroche, Syst\`emes dynamiques non commutatifs et moyennabilit\'e. Math. Ann. 279 (1987), 297--315.

\bibitem[ADR]{ADR} C.\,Anantharaman-Delaroche and J.\,Renault, ``Amenable groupoids'',  Monographies de L'Enseignement Math\'ematique, 36, 2000.

\bibitem[BaL]{BaL}S.\,Barlak and X.\,Li, Cartan subalgebras and the UCT problem. Adv. Math. 316 (2017), 748--769.

\bibitem[Bek$_1$]{BekkaForum}
B. Bekka, On the full C*-algebras of arithmetic groups and the congruence subgroup problem. Forum Math. 11 (1999), no. 6, 705--715.

\bibitem[Bek$_2$]{BekkaInv}
B. Bekka, Operator-algebraic superridigity for $SL_n(\Z)$, $n\geqslant 3$. Invent. Math. 169 (2007), no. 2, 401--425.

\bibitem[Bek$_3$]{BekkaPreprint} B. Bekka, On Bohr compactifications and profinite completions of group extensions. Math. Proc. Camb. Phil. Soc. 176 (2024), no. 2, 373--393.

\bibitem[BH]{BookUnitaryDual} B. Bekka and P. de la Harpe,`` Unitary representations of groups, duals, and characters,'' Mathematical Surveys and Monographs, 250. American Mathematical Society, Providence, 2020.

\bibitem[BHV]{BookPropT}
B. Bekka, P. de la Harpe and A. Valette,   ``Kazhdan's property (T),'' New Mathematical Monographs, 11. Cambridge University Press, Cambridge, 2008.

	
\bibitem[BLS]{BekkaInvariant} B. Bekka, A.T. Lau and G. Schlichting,
On invariant subalgebras of the Fourier-Stieltjes algebra of a locally compact group.
Math. Ann. 294 (1992), no. 3, 513--522.
	
\bibitem[BL]{BekkaLouvet} 	M.B.\,Bekka and N.\,Louvet, Some properties of C*-algebras associated to discrete linear groups. C*-algebras (M\"unster, 1999), 1--22, Springer, Berlin, 2000.

\bibitem[Bla]{Blackadar}
B. Blackadar, ``Operator Algebras: Theory of C*-Algebras and von Neumann Algebras,''
Encyclopaedia of Mathematical Sciences, vol. 122, Springer-Verlag, Berlin, 2006.



\bibitem[BrO]{BO}
N.\ P.\ Brown and N.\ Ozawa,  ``{$C^*$}-algebras and finite dimensional approximations,'' {\rm American Mathematical Society, 2008}.

\bibitem[C-SC]{Ceccherini}
T. Ceccherini-Silberstein and M. Coornaert, ``Cellular automata and groups,'' Springer, 2010.

\bibitem[Cho]{Choi}
M. Choi, The full C*-algebra of the free group on two generators. Pacific J. Math. 87(1980), 41--48.

\bibitem[CD-O]{CDO}
R. Clou\^{a}tre and A. Dor-On, Finite-dimensional approximations and semigroup coactions for operator algebras. Int. Math. Res. Not. 2024 (2024), no. 1, 698--744.



\bibitem[CS$_1$]{CS} K. Courtney and T. Shulman, Elements of C*-algebras attaining their norm in a finite-dimensional representation. Canad. J. Math. 71 (2019), no. 1, 93--111.

\bibitem[CS$_2$]{CSFDI}
K. Courtney and T. Shulman, Free products with amalgamation over central C*-subalgebras. Proc. Amer. Math. Soc. 148 (2020), no. 2, 765--776.

\bibitem[DF]{DF} F.\,Daalderop and R.\,Fokkink,
Chaotic homeomorphisms are generic,
Topology Appl. 102 (2000), no. 3, 297--302.

 \bibitem[Dad$_1$]{Dadarlat1}
 M.\,Dadarlat, Some remarks on the universal coefficient theorem in KK-theory, Operator
algebras and mathematical physics (Constanta, 2001), Theta, Bucharest, 2003, 65-74.

\bibitem[Dad$_2$]{Dadarlat2}
M.\,Dadarlat, Nonnuclear Subalgebras of AF Algebras,
American Journal of Mathematics
Vol. 122, No. 3, 2000,   581--597.



\bibitem[Dav]{Davidson}
K.\,Davidson,  ``{$C^*$}-algebras by example,''  Fields Institute Monographs, 6. American Mathematical Society, Providence, RI, 1996.



\bibitem[Dix]{Dixmier}
J.\,Dixmier,  ``{$C^*$}-algebras,''  North-Holland Mathematical Library, Vol. 15. North-Holland Publishing Co., Amsterdam-New York-Oxford, 1977.



\bibitem[ES]{ES}
C.\,Eckhardt and T.\,Shulman, On amenable Hilbert-Schmidt stable groups.  J. Funct. Anal. 285 (2023), no. 3, Paper No. 109954.

\bibitem[EL]{ExelLoring}
R.\,Exel and T.\,Loring,
Finite-dimensional representations of free product C*-algebras.
Internat. J. Math. 3 (1992), no. 4, 469--476.

\bibitem[Fe]{Fell}
J.\,Fell, The dual spaces of $C^*$-algebras.  Trans. Amer. Math. Soc. 94 (1960), 365--403.


\bibitem[FNT]{FNT}
T.\,Fritz, T.\,Netzer and A.\,Thom, Can you compute the operator norm? Proc. Amer. Math.
Soc. 142:4265-4276, 2014.

\bibitem[Gru]{Grunberg}
K.W.\, Gruenberg, Residual properties of infinite soluble groups. Proc. London Math. Soc. 7 (1957), 29--62.

\bibitem[Had$_1$]{HadwinNonseparable}
D.\,Hadwin, Nonseparable approximate equivalence. Trans. Amer. Math. Soc. 266 (1981), no.\,1, 203--231.


\bibitem[Had$_2$]{Hadwin}
D.\,Hadwin,  A lifting characterization of RFD C*-algebras. Math. Scand. 115 (2014),  no. 1,  85--95.


\bibitem[HS]{HS}
D.\,Hadwin and T.\,Shulman, Stability of group relations under small Hilbert-Schmidt perturbations. J. Funct. Anal. 275 (2018), no. 4, 761--792.

\bibitem[Hal]{Hall}
P.\,Hall,   The Frattini subgroups of finitely generated groups.  Proc. London Math. Soc. 11 (1961), 327--352.

\bibitem[Har]{Hartz}
M.\,Hartz, Finite dimensional approximations in operator algebras,
J. Funct.  Anal. 285 (2023), no.4, Paper 109974.

\bibitem[Hir]{Hirsch} K.A.\,Hirsch, On infinite soluble groups. III, Proc. London Math. Soc. 49 (1946), 184--194.

\bibitem[Jac]{Bhishan} B.\,Jacelon, Chaotic tracial dynamics. Forum of Mathematics Sigma, Vol.\ 11 : e39, 1-–21.

\bibitem[JNVWY]{CEP} Z.\,Ji, A.\,Natarajan, T.\,Vidick, J.\,Wright and H.\,Yuen, MIP*=RE, \emph{preprint} arXiv:2001.04383, 2020.

\bibitem[KL]{KL} D.\,Kielak and M.\,Linton, Virtually Free-by-Cyclic Groups. GAFA 34 (2024), 1580--1608.

\bibitem[LS]{LubotzkyShalom}
A.\,Lubotzky and Y.\,Shalom,  Finite representations in the unitary dual and Ramanujan groups, in `Discrete geometric analysis',
Contemp. Math., 347 (2004), pp. 173--189.

\bibitem[LZ]{LZ}
 A.\,Lubotzky and A.\,Zuk, ``On property tau'' (a preliminary version of a book, https://ma.huji.ac.il/~alexlub/BOOKS/On\%20property/On\%20property.pdf).

\bibitem[Nil]{Nilsen} M.\,Nilsen, The Stone-Čech compactification of Prim $A$. Bull. Austral. Math. Soc. 52 (1995), no. 3, 377--383.

\bibitem[Oza]{Taka}
N.\,Ozawa,
About the QWEP conjecture.
Internat. J. Math. 15 (2004), no. 5, 501--530.

\bibitem[OS]{TakaYuhei}
N.\,Ozawa and Y.\,Suzuki, On characterizations of amenable C*-dynamical systems and new examples. Selecta Math. (N.S.) 27 (2021), no. 5, Paper No. 92, 29 pp.


\bibitem[RW]{DanaIan}  I.\,Raeburn and D.P.\,Williams,
``Morita equivalence and continuous-trace C*-algebras,''
Mathematical Surveys and Monographs, 60. American Mathematical Society, Providence, RI, 1998.


\bibitem[Ren$_1$]{Re} J.\,Renault,  Cartan subalgebras in {$C^\ast$}-algebras. Irish Math. Soc. Bull. 61 (2008), 29--63.

\bibitem[Ren$_2$]{Renault}
J. Renault, Topological amenability is a Borel property. Math. Scand. 117 (2015), no. 1,
5--30.


\bibitem[Sim]{Aidan}
A.\,Sims, Hausdorff \'Etale Groupoids and Their C*-algebras, in
Operator algebras and dynamics: groupoids, crossed products, and Rokhlin dimension. Lecture notes from the Advanced Course held at Centre de Recerca Matem\'atica (CRM) Barcelona, March 13--17, 2017.


\bibitem[Shu]{Tatiana} T.\,Shulman,
Central amalgamation of groups and the RFD property. Adv. Math. 394 (2022), Paper No. 108131, 36 pp.


\bibitem[Tho]{Thoma}
E.\,Thoma,
\"Uber unit\"are Darstellungen abz\"ahlbarer, diskreter Gruppen. Math. Ann. 153 (1964), 111--138.

\bibitem[To$_1$]{TomiyamaBook} J.\,Tomiyama, ``Invitation to C*-algebras and topological dynamics,'' World Scientific Advanced Series in Dynamical Systems, 3. World Scientific Publishing Co., Singapore, 1987. x+167 pp.

\bibitem[To$_2$]{TomiyamaShort}
J.\,Tomiyama, C*-algebras and topological dynamical systems. Rev. Math. Phys. 8 (1996), no. 5, 741--760.

\bibitem[Voi]{Voic}
D.\,Voiculescu, A non-commutative Weyl-von Neumann theorem. Rev. Roumaine Math. Pures Appl. 21 (1976), no. 1, 97--113.

\bibitem[Wan]{Wang} S.P. Wang, On isolated points in the dual space of locally compact groups, Math. Ann.
218 (1975), 19--34.

\bibitem[Wil$_1$]{Dana1} D.P.\,Williams, ``Crossed products of C*-algebras,'' Mathematical Surveys and Monographs, 134. American Mathematical Society, Providence, RI, 2007.

\bibitem[Wil$_2$]{Dana} D.P.\,Williams, ``A tool kit for groupoid C*-algebras,'' Mathematical Surveys and Monographs, 241. American Mathematical Society, Providence, RI, 2019.

\bibitem[Will]{Rufus} R.\,Willett, A non-amenable groupoid whose maximal and reduced C*-algebras are the same.
M\"unster J. Math. 8 (2015), no. 1, 241--252.

	






\end{thebibliography}
\end{document}